\documentclass[12pt]{article}
\usepackage{amsthm,amsmath,amssymb,
extsizes,MnSymbol,easymat,mathrsfs,etex}
\usepackage{tikz}
\usetikzlibrary{matrix}

\usepackage[active]{srcltx}
\title{Miniversal
deformations of
matrices of bilinear
forms}
\date{}
\author{Andrii R. Dmytryshyn\thanks{Department
of Computing Science and HPC2N, Ume{\aa} University, SE-901 87 Ume{\aa},
Sweden. Email:
\mbox{andrii@cs.umu.se}.}
\and
Vyacheslav Futorny\thanks{Department of
Mathematics, University of S\~ao Paulo,
Brazil. Email:
\mbox{futorny@ime.usp.br}. Supported in
part by the CNPq grant (301743/2007-0)
and by the Fapesp grant
(2010/50347-9).}
    \and
Vladimir V. Sergeichuk\thanks{Corresponding author.
Institute of Mathematics, Tereshchenkivska 3,
Kiev, Ukraine. Email:
\mbox{sergeich@imath.kiev.ua}.
Supported in part by the Fapesp grants
(05/59407-6 and 2010/07278-6). The work was done while
this author was visiting the University
of S\~ao Paulo, whose hospitality is
gratefully acknowledged.}}

\renewcommand{\le}{\leqslant}
\renewcommand{\ge}{\geqslant}

\newtheorem{theorem}{Theorem}[section]
\newtheorem{lemma}{Lemma}[section]
\newtheorem{corollary}{Corollary}[section]

\theoremstyle{definition}
\newtheorem{definition}{Definition}[section]

\theoremstyle{remark}
\newtheorem{remark}{Remark}[section]
\newtheorem{example}{Example}[section]

\begin{document}
\maketitle

\begin{abstract}
V.I. Arnold [Russian Math. Surveys 26
(2) (1971) 29--43] constructed
miniversal deformations of square
complex matrices under similarity; that
is, a simple normal form to which not
only a given square matrix $A$ but all
matrices $B$ close to it can be reduced
by similarity transformations that
smoothly depend on the entries of $B$.
We construct miniversal deformations of
matrices under congruence.

{\it AMS classification:} 15A21

{\it Keywords:} Bilinear forms;
Miniversal deformations
\end{abstract}

\section{Introduction}
\label{introd}

The reduction of a matrix to its Jordan
form is an unstable operation: both the
Jordan form and a reduction
transformation depend discontinuously
on the entries of the original matrix.
Therefore, if the entries of a matrix
are known only approximately, then it
is unwise to reduce it to Jordan form.
Furthermore, when investigating a
family of matrices smoothly depending
on parameters, then although each
individual matrix can be reduced to its
Jordan form, it is unwise to do so
since in such an operation the
smoothness relative to the parameters
is lost.

For these reasons, V.I. Arnold
\cite{arn} (see also \cite{arn2,arn3})
constructed miniversal deformations of
matrices under similarity; that is, a
simple normal form to which not only a
given square matrix $A$ but all
matrices $B$ close to it can be reduced
by similarity transformations that
smoothly depend on the entries of $B$.
Miniversal deformations were also
constructed for:
\begin{itemize}
  \item real matrices with respect
      to similarity \cite{gal} (see
      also \cite{arn2,arn3}; this
      normal form was simplified in
      \cite{gar_ser});

  \item complex matrix pencils
      \cite{kag} (i.e., matrices of
      pairs of linear mappings
      $U\rightrightarrows V$; other
      normal forms of complex and
      real matrix pencils were
      constructed in
      \cite{gar_ser,k-s_triang},
      see also \cite{gar_ser1});

  \item complex and real
      contragredient matrix pencils
      \cite{gar_ser} (i.e.,
      matrices of pairs of counter
      linear mappings
      $U\rightleftarrows V$);

   \item matrices of selfadjoint
       operators on a complex or
       real vector space with
       scalar product given by a
       skew-symmetric, or
       symmetric, or Hermitian
       nonsingular form, see
       \cite{gal2,djo,pat1,pat3}
       and \cite[Appendix
       6]{arn_mex};

  \item matrices of linear
      operators on a unitary space
     \cite{ben}. Deformations of
     selfadjoint operators
     (Hermitian forms) on a unitary
     space are studied in
     \cite{von}.

\end{itemize}

All matrices that we consider are
complex matrices.

In Section \ref{s2}, we formulate
Theorem \ref{teo2} that gives
miniversal deformations of matrices of
bilinear forms; i.e., miniversal
deformations of matrices with respect
to \emph{congruence transformations}
\begin{equation*}\label{djs}
A\mapsto S^TAS,\qquad S \text{ is
nonsingular}
\end{equation*}
(and hence miniversal deformations of
pairs consisting of a symmetric matrix
and a skew-symmetric matrix since each
square matrix can be expressed uniquely
as their sum; see Remark \ref{HRe}). A
more abstract form of Theorem
\ref{teo2}, in the spirit of Arnold's
article \cite{arn}, is given in Theorem
\ref{teojy} of Section \ref{sect2}.

We prove Theorem \ref{teojy} in
Sections \ref{sect3}--\ref{s7}. The
proof is based on Lemma \ref{t2.1},
which gives a method for constructing
miniversal deformations. This lemma
follows from a general theory of
miniversal deformations. In Section
\ref{sect4} we give its constructive
proof and find a congruence
transformation that reduces a matrix to
its miniversal deformation. Analogous
interactive methods for constructing
transforming matrices in the reduction
to versal deformations of matrices
under similarity and of matrix pencils
under equivalence were developed in
\cite{gar_mai,mai,mai1}.

A preliminary version of this article
appeared in 2007 preprint \cite{f_s};
it was used in \cite{f_s_3x3} for
constructing the Hasse diagram of the
closure ordering on the set of
congruence classes of $3\times 3$
matrices. The authors also recently
obtained miniversal deformations of
matrices of
\begin{itemize}
  \item sesquilinear forms
      \cite{def-sesq} (which allows
      to construct miniversal
deformations of pairs $(H_1,H_2)$
of Hermitian matrices because each
square matrix can be expressed
uniquely as their sum $H_1+iH_2$),
  \item pairs of skew-symmetric
      forms \cite{dm1}, and
  \item pairs of symmetric forms
      \cite{dm2}.
\end{itemize}

\section{The main
theorem in terms of holomorphic matrix
functions}\label{s2}

Define the $n\times n$ matrices:
\begin{equation*}\label{1aa}
J_n(\lambda):=\begin{bmatrix}
\lambda&1&&0\\&\lambda&\ddots&\\&&\ddots&1
\\ 0&&&\lambda
\end{bmatrix},\qquad
\Gamma_n :=
\begin{bmatrix} 0&&&&
\udots
\\&&&-1&\udots
\\&&1&1\\ &-1&-1& &\\
1&1&&&0
\end{bmatrix}.
\end{equation*}

We use the following canonical form of
complex matrices for congruence.

\begin{theorem}[{\cite{hor-ser}}]
\label{t2a}  Each square complex matrix
is congruent to a direct sum,
determined uniquely up to permutation
of summands, of matrices of the form
\begin{equation}
\label{can}
H_m(\lambda):=
\begin{bmatrix}0&I_m\\
J_m(\lambda) &0
\end{bmatrix}\ (\lambda \ne 0,\ \lambda\ne
(-1)^{m+1}),
 \quad
\Gamma_n,
 \quad
J_k(0)
  \end{equation}
in which $\lambda\in\mathbb C$ is
determined up to replacement by
$\lambda^{-1}$.
\end{theorem}
This canonical form was obtained in
\cite{hor-ser} basing on \cite[Theorem
3]{ser_izv} and was generalized to
other fields in
\cite{hor-ser_anyfield}; a direct proof
that this form is canonical was given
in \cite{hor-ser_regul, hor-ser_can}.

Let
\begin{equation}\label{kus}
A_{\text{can}}=
\bigoplus_i
H_{p_i}(\lambda_i)
 \oplus
\bigoplus_j\Gamma_{q_j}
 \oplus
\bigoplus_l
J_{r_l}(0),\qquad
r_1\ge r_2\ge\dots,
\end{equation}
be the canonical form for congruence of
an $n\times n$ matrix $A$. Then
\begin{equation}\label{gdr}
S^TAS=A_{\text{can}}
\end{equation}
for a nonsingular $S$. All matrices
that are close to $A$ are represented
in the form $A +E $ in which
$E\in\mathbb C^{n\times n}$ is close to
$0_n$.

Let ${\cal S}(E)$ be an $n\times n$
matrix function that is holomorphic on
a neighborhood of $0_n$, which means
that ${\cal S}(E)$ is an $n\times n$
matrix whose entries are power series
in $n^2$ entries of $E$, and these
series are convergent in this
neighborhood of $0_n$. Let ${\cal
S}(0_n)=S$ in which $S$ is from
\eqref{gdr}. We define the matrix
function ${\cal D}(E)$ by
\begin{equation}\label{ksy}
A_{\text{can}} +{\cal
D}(E) ={\cal S}(E)^T
(A+E) {\cal
S}(E).
\end{equation}
Then ${\cal D}(E)$ is holomorphic at
$0_n$ and ${\cal D}(0_n)=0_n$. Our
purpose is to find a simple form of
${\cal D}(E)$ by choosing a suitable
${\cal S}(E)$. In Theorem \ref{teo2},
we give ${\cal D}(E)$ with the minimal
number of nonzero entries that can be
attained by using transformations
\eqref{ksy}.

By a \emph{$(0,\!*)$ matrix} we mean a
matrix whose entries are $0$ and $*$.
Theorem \ref{teo2} involves the
following $(0,\!*)$ matrices, in which
all stars are placed in one row or
column:

$\bullet$ The $m\times n$ matrices
\[0^{\nwarrow}:=\left[\begin{MAT}(e){cccc}
*&&&\\
\vdots&&0&\\
*&&&\\
\end{MAT}\right]\text{ if $m\le n$, or }
\left[\begin{MAT}(e){ccc}
*&\cdots&*\\
&&\\
&0&\\
&&\\
\end{MAT}\right] \text{ if $m\ge n$},
\]
\[0^{\nwvdash}:=\left[\begin{MAT}(b){cccccccc}
*&&&&&&&\\
0&&&&&&&\\
*&&&&0&&&\\
0&&&&&&&\\
\vdots&&&&&&&
\\
\end{MAT}\right]\text{ if $m\le n$, or }
\left[\begin{MAT}(e){ccccc}
*&0&*&0&\cdots\\
&&&&\\
&&&&\\
&&0&&\\
&&&&\\
&&&&\\
&&&&\\
\end{MAT}\right]\text{ if $m\ge n$},
\]
\[0^{\nwmodels}:=\left[\begin{MAT}(b){cccccccc}
0&&&&&&&\\
*&&&&&&&\\
0&&&&0&&&\\
*&&&&&&&\\
\vdots&&&&&&&
\\
\end{MAT}\right]\text{ if $m\le n$, or }
\left[\begin{MAT}(e){ccccc}
0&*&0&*&\cdots\\
&&&&\\
&&&&\\
&&0&&\\
&&&&\\
&&&&\\
&&&&
\\
\end{MAT}\right]\text{ if $m\ge n$}
\]
(if $m=n$ then we can use both the left
and the right matrix).

$\bullet$ The matrices
\[0^{\nearrow}, 0^{\nevdash},
0^{\nemodels};\qquad 0^{\searrow},
0^{\sevdash}, 0^{\semodels};\qquad
0^{\swarrow}, 0^{\swvdash},
0^{\swmodels}\] are obtained from
$0^{\nwarrow}, 0^{\nwvdash},
0^{\nwmodels}$ by the clockwise
rotation through $90^{\circ}$;
respectively, $180^{\circ}$; and
$270^{\circ}$.

$\bullet$ The $m\times n$ matrices
\begin{equation*}\label{bjhf}
0^{\updownarrow}:=\left[\begin{MAT}(b){ccc}
*&\cdots&*\\
&&\\
&0&\\
&&\\
\end{MAT}\right]\quad \text{or}\quad
\left[\begin{MAT}(b){ccc}
&&\\
&0&\\
&&\\
*&\cdots&*\\
\end{MAT}\right]\end{equation*}
($0^{\updownarrow}$ can be taken in any
of these forms), and
\begin{equation}\label{hui}
{\cal P}_{mn}:=\begin{bmatrix}
\begin{matrix}
0&\dots& 0
  \\
\vdots&\ddots& \vdots
\end{matrix}&0\\
\begin{matrix}
0& \dots&
0\end{matrix}&
\begin{matrix}
0\ *\ \dots\ *
\end{matrix}
\end{bmatrix}\quad
\text{in which } m\le n
\end{equation}
(${\cal P}_{mn}$ has $n-m-1$ stars if
$m<n$).

Let $A_{\text{can}}=A_1\oplus
A_2\oplus\dots\oplus A_t$ be the
decomposition \eqref{kus}, and let
${\cal D}(E)$ in \eqref{ksy} be
partitioned conformably to the
partition of $A_{\text{can}}$:

\begin{equation}\label{grsd}
{\cal D}={\cal D}(E)
=\begin{bmatrix}
{\cal
D}_{11}&\dots&{\cal
D}_{1t}
 \\
\vdots&\ddots&\vdots\\
{\cal
D}_{t1}&\dots&{\cal
D}_{tt}
\end{bmatrix}.
\end{equation}
Write
\begin{equation}\label{lhs}
{\cal D}(A_i):={\cal
D}_{ii},\qquad {\cal
D}(A_i,A_j) :=({\cal
D}_{ji},{\cal
D}_{ij})\ \ \text{if
}i<j.
\end{equation}

Our main result is the following
theorem, which we reformulate in a more
abstract form in Theorem \ref{teojy}.

\begin{theorem}[{\cite{f_s}}]\label{teo2}
Let $A$ be a square complex matrix, let
$A_{\text{can}}$  be its canonical
matrix \eqref{kus} for congruence, and
let $S$ be a nonsingular matrix such
that $S^TAS=A_{\text{can}}$. Then all
matrices $A+E$ that are sufficiently
close to $A$ can be simultaneously
reduced by some transformation
\begin{equation}\label{tef}
A+E\mapsto {\cal
S}(E)^T (A+E) {\cal
S}(E),\quad\begin{matrix}
\text{${\cal S}(E)$
is nonsingular and}\\
\text{holomorphic at zero, } {\cal
S}(0)=S
\end{matrix}
\end{equation}
to the form $A_{\text{can}} +{\cal D}$
in which ${\cal D}$ is a $(0,\!*)$
matrix whose stars represent entries
that depend holomorphically on the
entries of $E$, the number of stars in
${\cal D}$ is minimal that can be
achieved by transformations of the form
\eqref{tef}, and the blocks of ${\cal
D}$ with respect to the partition
\eqref{grsd} are defined in the
notation \eqref{lhs} as follows:

\begin{itemize}
  \item[{\rm(i)}] The diagonal
      blocks of ${\cal D}$ are
      defined by
\begin{equation}\label{KEV}
{\cal
D}(H_{m}(\lambda))=
  \begin{cases}
    \begin{bmatrix}
0&0
 \\ 0^{\swarrow}&0
\end{bmatrix}
 &\text{if $\lambda\ne\pm 1$ $($all blocks are $m\times m)$},
                           \\[5mm]
 \begin{bmatrix}
0^{\nwmodels}&0
 \\ 0^{\swarrow}&0^{\semodels}
\end{bmatrix}
 &\text{if $\lambda=1$
 $(m$ is even by \eqref{can}$)$},
                      \\[5mm]
 \begin{bmatrix}
0^{\nwvdash}&0
 \\ 0^{\swarrow}&0^{\sevdash}
\end{bmatrix}
 &\text{if $\lambda=-1$
 $(m$ is odd by \eqref{can}$)$;}
  \end{cases}
\end{equation}
\begin{equation} \label{uwm}
{\cal D} (\Gamma_n)=
  \begin{cases}
0^{\nwvdash}
 & \text{if $n$ is even}, \\
0^{\nwmodels}
& \text{if $n$ is
odd};
  \end{cases}
\end{equation}
\begin{equation}\label{lsiu}
{\cal
D}(J_n(0))=0^{\swvdash}.
\end{equation}

  \item[{\rm(ii)}] The off-diagonal
      blocks of ${\cal D}$ whose
      horizontal and vertical
      strips contain summands of
      $A_{\text{can}}$ of the same
      type are defined by
\begin{equation}\label{lsiu1}
{\cal D}
(H_m(\lambda),\,
H_n(\mu))
       =
  \begin{cases}
(0,\:0) &\text{if
$\lambda\ne\mu^{\pm
1}$},
      \\
    \left(\begin{bmatrix}
0^{\nwarrow}&0
 \\ 0&0^{\searrow}
\end{bmatrix},\:0
\right)
 &\text{if $\lambda=\mu^{-1}\ne\pm 1$},
                           \\[5mm]
    \left(\begin{bmatrix}
0&0^{\nearrow}
 \\ 0^{\swarrow}&0
\end{bmatrix},\:0
\right)
 &\text{if $\lambda=\mu\ne\pm 1$},
                           \\[5mm]
    \left(\begin{bmatrix}
0^{\nwarrow}&0^{\nearrow}
 \\ 0^{\swarrow}&0^{\searrow}
\end{bmatrix},\:0
\right)
 &\text{if $\lambda=\mu=\pm
 1$};
  \end{cases}
\end{equation}
\begin{equation}\label{lsiu2}
{\cal D}
(\Gamma_m,\Gamma_n)=
  \begin{cases}
(0,\: 0)
 & \text{if $m-n$ is odd}, \\
(0^{\nwarrow},\:0)
 & \text{if $m-n$ is even};
  \end{cases}
\end{equation}
\begin{equation}\label{lsiu3}
{\cal D}
(J_m(0),J_n(0))=
  \begin{cases}
(0^{\swvdash},\:
0^{\swvdash})
 & \text{if $m\ge n$ and $n$ is even},
    \\
(
0^{\swvdash}+{\cal
P}_{nm},\:0^{\swvdash})
 & \text{if $m\ge n$ and $n$ is
 odd}.
  \end{cases}
\end{equation}

  \item[{\rm(iii)}] The
      off-diagonal blocks of ${\cal
      D}$ whose horizontal and
      vertical strips contain
      summands of $A_{\text{can}}$
      of different types are
      defined by
\begin{equation}\label{lsiu4}
{\cal D}
(H_m(\lambda),\Gamma_n)=
\begin{cases} (0,\: 0)
 & \text{if
 $\lambda\ne(-1)^{n+1}$}, \\
(
[0^{\nwarrow}\
0^{\nearrow}],\: 0)
 & \text{if
 $\lambda=(-1)^{n+1}$};
  \end{cases}
 \end{equation}
\begin{equation}\label{lsiu5}
{\cal D}
(H_m(\lambda),J_n(0))=
  \begin{cases}
(0,\: 0)
 & \text{if $n$ is even}, \\
(
0^{\updownarrow},\: 0)
 & \text{if $n$ is odd};
  \end{cases}
\end{equation}
\begin{equation}\label{lsiu6}
{\cal D}
(\Gamma_m,J_n(0))=
  \begin{cases}
(0,\: 0)
 & \text{if $n$ is even}, \\
(
0^{\updownarrow},\: 0)
 & \text{if $n$ is
 odd}.
  \end{cases}
 \end{equation}
\end{itemize}
\end{theorem}

For each $A\in{\mathbb C}^{n\times n}$,
the vector space
\begin{equation}\label{eelie}
T(A):=\{C^TA+AC\,|\,C\in{\mathbb
C}^{n\times n}\}
\end{equation}
is the tangent space to the congruence
class of $A$ at the point $A$ since
\begin{equation}\label{fwt}
(I+\varepsilon
C)^TA(I+\varepsilon C)
=A+\varepsilon(C^TA+AC)
+\varepsilon^2C^TAC
\end{equation}
for all $C\in{\mathbb C}^{n\times n}$
and $\varepsilon\in\mathbb C$.

The matrix $\cal D$ from Theorem
\ref{teo2} was constructed such that
\begin{equation}\label{jyr}
{\mathbb C}^{\,n\times
n}=T(A_{\text{can}}) \oplus {\cal
D}({\mathbb C})
\end{equation}
in which ${\cal D}({\mathbb C})$ is the
vector space of all matrices obtained
from $\cal D$ by replacing its stars by
complex numbers. Thus, the number of
stars in $\cal D$ is equal to the
codimension of the congruence class of
$A_{\text{can}}$; it was independently
calculated in \cite{t_d}. The
codimensions of *congruence classes of
canonical matrices for *congruence were
calculated in \cite{t_d_*}. Simplest
miniversal deformations of matrix
pencils and contagredient matrix
pencils \cite{gar_ser}, canonical
matrices for *congruence
\cite{def-sesq}, and canonical pairs of
skew-symmetric matrices \cite{dm2} were
constructed by analogous methods.

\emph{Theorem \ref{teo2} will be proved
as follows:} we first prove in Lemma
\ref{t2.1} that each $(0,\!*)$ matrix
that satisfies \eqref{jyr} can be taken
as $\cal D$ in Theorem \ref{teo2}, and
then verify that $\cal D$ from Theorem
\ref{teo2} satisfies \eqref{jyr}.

\begin{example}
    \label{colh}
Let $A$ be any $2\times 2$ or $3\times
3$ matrix. Then all matrices $A+E$ that
are sufficiently close to $A$ can be
simultaneously reduced by
transformations \eqref{tef} to one of
the following forms
\begin{align*}
                           &
\begin{bmatrix} 0&\\
&0
 \end{bmatrix}+
\begin{bmatrix}
 *&*\\ *&*
 \end{bmatrix},
                            &&
\begin{bmatrix} 1&\\
&0
 \end{bmatrix}+
\begin{bmatrix}
 0&0\\ *&*
 \end{bmatrix},
                         &&
\begin{bmatrix} 1&\\
&1
 \end{bmatrix}+
\begin{bmatrix}
 0&0\\ *&0
 \end{bmatrix},
                         \\ &
\begin{bmatrix} 0&1\\
-1&0
 \end{bmatrix}+
\begin{bmatrix}
 *&0\\ *&*
 \end{bmatrix},
                         &&
\begin{bmatrix} 0&1\\
\lambda &0
 \end{bmatrix}+
\begin{bmatrix}
 0&0\\ *&0
 \end{bmatrix}\ (\lambda\ne\pm 1),
                        &&
\begin{bmatrix} 0&-1\\
1&1
 \end{bmatrix}+
\begin{bmatrix}
 *&0\\ 0&0
 \end{bmatrix},
\end{align*}
or, respectively,
\begin{align*}
&\begin{bmatrix} 0&&\\
&0&\\&&0
 \end{bmatrix}+
\begin{bmatrix}
 *&*&*\\  *&*&*\\ *&*&*
 \end{bmatrix},
                             &&
\begin{bmatrix} 1&&\\
&0&\\&&0
 \end{bmatrix}+
\begin{bmatrix}
 0&0&0\\  *&*&*\\ *&*&*
 \end{bmatrix},
                            \\&
\begin{bmatrix} 1&&\\
&1&\\&&0
 \end{bmatrix}+
\begin{bmatrix}
 0&0&0\\  *&0&0\\ *&*&*
 \end{bmatrix},
                           &&
\begin{bmatrix}
1&&\\ &1&\\&&1
 \end{bmatrix}+
\begin{bmatrix}
 0&0&0\\  *&0&0\\
 *&*&0
 \end{bmatrix},
                        \\&
\begin{bmatrix} 0&1&\\
-1&0&\\&&0
 \end{bmatrix}+
\begin{bmatrix}
 *&0&0\\  *&*&0\\ *&*&*
 \end{bmatrix},
                       &&
\begin{bmatrix}
0&1&\\ \lambda
&0&\\&&0
 \end{bmatrix}+
\begin{bmatrix}
 0&0&0\\  *&0&0\\
 *&*&*
 \end{bmatrix}\ (\lambda\ne0,\pm 1),
                      \\&
\begin{bmatrix}
0&1&\\ 0 &0&\\&&0
 \end{bmatrix}+
\begin{bmatrix}
 0&0&0\\  *&0&*\\
 *&0&*
 \end{bmatrix},
                          &&
\begin{bmatrix}
0&-1&\\ 1&1&\\&&0
 \end{bmatrix}+
\begin{bmatrix}
 *&0&0\\  0&0&0\\ *&*&*
 \end{bmatrix},
                           \\&
\begin{bmatrix} 0&1&\\
-1&0&\\&&1
 \end{bmatrix}+
\begin{bmatrix}
 *&0&0\\  *&*&0\\
 0&0&0
 \end{bmatrix},
                         &&
\begin{bmatrix}
0&1&\\ \mu &0&\\&&1
 \end{bmatrix}+
\begin{bmatrix}
 0&0&0\\  *&0&0\\
 0&0&0
 \end{bmatrix}\ (\mu\ne \pm 1),
                     \\&
 \begin{bmatrix}
0&-1&\\ 1&1&\\&&1
 \end{bmatrix}+
\begin{bmatrix}
 *&0&0\\  0&0&0\\
 0&0&0
 \end{bmatrix},
                      &&
\begin{bmatrix}
0&1&0\\ 0&0&1\\0&0&0
 \end{bmatrix}+
\begin{bmatrix}
 0&0&0\\  0&0&0\\
 *&0&*
 \end{bmatrix},
                     \\&
 \begin{bmatrix}
0&0&1\\ 0&-1&-1\\1&1&0
 \end{bmatrix}+
\begin{bmatrix}
 0&0&0\\  *&0&0\\
 0&0&0
 \end{bmatrix}.
\end{align*}

Each of these matrices has the form
$A_{\rm can}+{\cal D}$ in which $A_{\rm
can}$ is a direct sum of blocks of the
form \eqref{can} (the zero entries
outside of these blocks in $A_{\rm
can}$ are not shown) and the stars in
${\cal D}$ are complex numbers that
tend to zero as $E$ tends to $0$. The
number of stars is the smallest that
can be attained by using
transformations \eqref{tef}; it is
equal to the codimension of the
congruence class of $A$.
\end{example}

\section{The main theorem in terms of
miniversal deformations} \label{sect2}

The notion of a miniversal deformation
of a matrix with respect similarity was
given by Arnold \cite{arn} (see also
\cite[\S\,30B]{arn3}). It can be
extended to matrices with respect to
congruence as follows.

A \emph{deformation} of a matrix
$A\in{\mathbb C}^{n\times n}$ is a
holomorphic map ${\cal A}: \Lambda\to
{\mathbb C}^{n\times n}$ in which
$\Lambda\subset \mathbb C^k$ is a
neighborhood of $\vec 0=(0,\dots,0)$
and ${\cal A}(\vec 0)=A$.

Let ${\cal A}$ and ${\cal B}$ be two
deformations of $A$ with the same
parameter space $\mathbb C^k$. We
consider ${\cal A}$ and ${\cal B}$ as
\emph{equal} if they coincide on some
neighborhood of $\vec 0$ (this means
that each deformation is a germ). We
say that ${\cal A}$ and ${\cal B}$ are
\emph{equivalent} if the identity
matrix $I_n$ possesses a deformation
${\cal I}$ such that
\begin{equation*}\label{kft}
{\cal B}(\vec\lambda)=
{\cal
I}(\vec\lambda)^{T}
{\cal A}(\vec\lambda)
{\cal I}(\vec\lambda)
\end{equation*}
for all $\vec\lambda=(\lambda_1,\dots,
\lambda_k)$ in some neighborhood of
$\vec 0$.

\begin{definition}\label{d}
A deformation ${\cal
A}(\lambda_1,\dots,\lambda_k)$ of a
square matrix $A$ is called
\emph{versal} if every deformation
${\cal B}(\mu_1,\dots,\mu_l)$ of $A$ is
equivalent to a deformation of the form
${\cal A}(\varphi_1(\vec\mu),\dots,
\varphi_k(\vec\mu))$ in which
$\vec\mu=(\mu_1,\dots,\mu_l)$, all
$\varphi_i(\vec\mu)$ are power series
that are convergent in a neighborhood
of $\vec 0$, and $\varphi_i(\vec 0)=0$.
A versal deformation ${\cal
A}(\lambda_1,\dots,\lambda_k)$ of $A$
is called \emph{miniversal} if there is
no versal deformation that has less
than $k$ parameters.
\end{definition}

For each $(0,\!*)$ matrix ${\cal D}$,
we denote by ${\cal D}(\mathbb C)$ the
space of all matrices obtained from
$\cal D$ by replacing the stars with
complex numbers (as in \eqref{jyr}) and
 by ${\cal D}(\vec
{\varepsilon})$ the parameter matrix
obtained from ${\cal D}$ by replacing
each $(i,j)$ star with the parameter
${\varepsilon}_{ij}$. This means that
\begin{equation}\label{a2z}
{\cal D}(\mathbb C):=
\bigoplus_{(i,j)\in{\cal
I}({\cal D})} {\mathbb
C} E_{ij}, \qquad
{\cal D}(\vec
{\varepsilon}):=
\sum_{(i,j)\in{\cal
I}({\cal D})}
\varepsilon_{ij}E_{ij},
\end{equation}
in which every $E_{ij}$ is the matrix
unit (its $(i,j)$ entry is $1$ and the
others are $0$) and
\begin{equation*}\label{a2za}
{\cal I}({\cal
D})\subseteq
\{1,\dots,n\}\times
\{1,\dots,n\}
\end{equation*}
is the set of indices of the stars in
${\cal D}$.

We say that a deformation of $A$ is
\emph{simplest} if it has the form
$A+{\cal D}(\vec {\varepsilon})$ in
which $\cal D$ is a $(0,\!*)$ matrix.
Definition \ref{d} of versality for a
simplest deformation can be
reformulated in the spirit of Section
\ref{s2} as follows.

\begin{definition}\label{dver}
A simplest deformation $A+{\cal D}(\vec
{\varepsilon})$ of a square matrix $A$
is \emph{versal} if there exists an
$n\times n$ matrix ${\cal S}(X)$ and a
neighborhood $U\subset \mathbb
C^{n\times n}$ of $0_n$ such that
\begin{itemize}
  \item[(i)] the entries of ${\cal
      S}(X)$ are power series in
      variables $x_{ij}$,
      $i,j=1,\dots,n$ (they form
      the $n\times n$ matrix of
      unknowns $X=[x_{ij}]$),
  \item[(ii)] these series are
      convergent in $U$ and ${\cal
      S}(0_n)=I_n,$
  \item[(iii)] ${\cal
      S}(E)^T(A+E){\cal S}(E)\in
      A+{\cal D}(\mathbb C)$ for
      all $E\in U$.
\end{itemize}
\end{definition}

Since each square matrix is congruent
to its canonical matrix, it suffices to
construct miniversal deformations of
canonical matrices \eqref{kus}. Their
miniversal deformations are given in
the following theorem, which is another
form of Theorem \ref{teo2}.

\begin{theorem}[{\cite{f_s}}]\label{teojy}
Let $A_{\text{can}}$ be a canonical
matrix \eqref{kus} for congruence. A
simplest miniversal deformation of
$A_{\text{can}}$ can be taken in the
form $A_{\text{can}} +{\cal D}(\vec
{\varepsilon})$, where $\cal D$ is the
$(0,\!*)$ matrix partitioned into
blocks ${\cal D}_{ij}$ $($as in
\eqref{grsd}$)$ that are defined by
\eqref{KEV}--\eqref{lsiu6} in the
notation \eqref{lhs}.
\end{theorem}

\begin{remark}\label{HRe}
Each square matrix $A$ can be
represented uniquely as
\begin{equation}\label{fsh}
A=\mathscr{S}+
\mathscr{C},\quad\text{$\mathscr{S}$ is
symmetric and $\mathscr{C}$ is
skew-symmetric.}
\end{equation}
A congruence of $A$ corresponds to a
simultaneous congruence of
$\mathscr{S}$  and $\mathscr{C}$. Thus,
if $A_{\text{can}}$ is a canonical
matrix for congruence given in Theorem
\ref{t2a} and
$A_{\text{can}}=\mathscr{S}_{\text{can}}+
\mathscr{C}_{\text{can}}$ is its
representation \eqref{fsh}, then
$(\mathscr{S}_{\text{can}},
\mathscr{C}_{\text{can}})$ is a
canonical pair for simultaneous
congruence of pairs of symmetric and
skew-symmetric matrices. The pairs
$(\mathscr{S}_{\text{can}},
\mathscr{C}_{\text{can}})$ were
described in \cite[Theorem
1.2(a)]{hor-ser_can}. Theorem
\ref{teojy} admits to derive a
miniversal deformation of
$(\mathscr{S}_{\text{can}},
\mathscr{C}_{\text{can}})$; that is, to
construct a normal form with minimal
number of parameters to which all pairs
$(\mathscr{S},\mathscr{C})$ that are
close to $(\mathscr{S}_{\text{can}},
\mathscr{C}_{\text{can}})$ and consist
of symmetric and skew-symmetric
matrices can be reduced by
transformations
\[
(\mathscr{S},\mathscr{C})\mapsto
(S^T\mathscr{S}S,\,S^T\mathscr{C}S),
\qquad S\text{ is nonsingular},
\]
in which $S$ smoothly depends on the
entries of $\mathscr{S}$ and
$\mathscr{C}$. All one has to do is to
express $A_{\text{can}} +{\cal D}(\vec
{\varepsilon})$ as the sum of symmetric
and skew-symmetric matrices.
\end{remark}

\section{A method for constructing
miniversal deformations} \label{sect3}

In this section, we give a method for
constructing simplest miniversal
deformations; we use it in the proof of
Theorem \ref{teojy}.

The deformation
\begin{equation}\label{edr} {\cal
U}(\vec {\varepsilon}):=
A+\sum_{i,j=1}^n
\varepsilon_{ij}E_{ij},
\end{equation} in which $E_{ij}$
are the matrix units,
is universal in the sense that every
deformation ${\cal
B}(\mu_1,\dots,\mu_l)$ of $A$ has the
form ${\cal U}(\vec{\varphi}
(\mu_1,\dots,\mu_l)),$ in which
$\varphi_{ij}(\mu_1,\dots,\mu_l)$ are
power series that are convergent in a
neighborhood of $\vec 0$ and
$\varphi_{ij}(\vec 0)= 0$. Hence every
deformation ${\cal
B}(\mu_1,\dots,\mu_l)$ in Definition
\ref{d} can be replaced by ${\cal
U}(\vec {\varepsilon})$, which gives
the following lemma.

\begin{lemma}\label{lem}
The following two conditions are
equivalent for any deformation ${\cal
A}(\lambda_1,\dots,\lambda_k)$ of a
matrix $A$:
\begin{itemize}
  \item[\rm(i)] The deformation
      ${\cal
      A}(\lambda_1,\dots,\lambda_k)$
      is versal.
  \item[\rm(ii)] The deformation
      \eqref{edr} is equivalent to
      ${\cal
      A}(\varphi_1(\vec{\varepsilon}),\dots,
      \varphi_k(\vec{\varepsilon}))$
      for some power series
      $\varphi_i(\vec{\varepsilon})$
      that are convergent in a
      neighborhood of\/ $\vec 0$
      and such that $\varphi_i(\vec
      0)=0$.
\end{itemize}
\end{lemma}
If $U$ is a subspace of a vector space
$V$, then each set $v+U$ with $v\in V$
is called an \emph{affine subspace
parallel to $U$}.

The proof of Theorem \ref{teojy} is
based on the following lemma, which
gives a method of constructing
miniversal deformations. A constructive
proof of this lemma is given in Theorem
\ref{ttft}.

\begin{lemma}
 \label{t2.1}
Let $A\in{\mathbb C}^{\,n\times n}$ and
let $\cal D$ be a $(0,\!*)$ matrix of
size $n\times n$. The following three
statements are equivalent:
\begin{itemize}
  \item[\rm(i)] The deformation
      $A+{\cal D}(\vec
      {\varepsilon})$ of $A\ ($see
      \eqref{a2z}$)$ is miniversal.

  \item[\rm(ii)] The vector space
      ${\mathbb C}^{\,n\times n}$
      decomposes into the direct
      sum
\begin{equation*}\label{a4}
{\mathbb C}^{\,n\times
n}=T(A)
\oplus {\cal
D}({\mathbb C})
\end{equation*}
in which $T(A)$ and ${\cal
D}({\mathbb C})$ are defined in
\eqref{eelie} and \eqref{a2z}.

  \item[\rm(iii)] Each affine
      subspace of ${\mathbb
      C}^{\,n\times n}$ parallel to
      $T(A)$ intersects ${\cal
      D}(\mathbb C)$ at exactly one
      point.
\end{itemize}
\end{lemma}

\begin{proof} Let us define the action
of the group
$GL_n(\mathbb C)$ of nonsingular
$n\times n$ matrices on the space
${\mathbb C}^{\,n\times n}$ by
\begin{equation}\label{hhl}
A^S:=S^T AS,\qquad
A\in{\mathbb
C}^{\,n\times n},\
S\in GL_n(\mathbb C).
\end{equation}
The orbit $A^{GL_n}$ of $A$ under this
action consists of all matrices that
are congruent to $A$.

By \eqref{fwt}, the space $T(A)$ is the
tangent space to the orbit $A^{GL_n}$
at the point $A$. Hence ${\cal D}(\vec
{\varepsilon})$ is transversal to the
orbit $A^{GL_n}$ at the point $A$ if
\[
{\mathbb C}^{\,n\times
n}=T(A) +
{\cal D}({\mathbb C})
\]
(see definitions in
\cite[\S\,29E]{arn3}; two subspaces of
a vector space are called
\emph{transversal} if their sum is the
whole space).

This proves the equivalence of (i) and
(ii) since a transversal (of the
minimal dimension) to the orbit is a
(mini)versal deformation; see
\cite[Section 1.6]{arn2} or \cite[Part
V, Theorem 1.2]{tan}. The equivalence
of (ii) and (iii) is obvious.
\end{proof}

Recall that the orbits of canonical
matrices \eqref{kus} under the action
\eqref{hhl} were also studied in
\cite{t_d,f_s_3x3}.

\begin{corollary}
A simplest miniversal deformation of
$A\in{\mathbb C}^{\,n\times n}$ can be
constructed as follows. Let
$T_1,\dots,T_r$ be a basis of the space
$T(A)$, and let $E_1,\dots,E_{n^2}$ be
the basis of ${\mathbb C}^{\,n\times
n}$ consisting of all matrix units
$E_{ij}$. Removing from the sequence
$T_1,\dots, T_r,E_1,\dots,E_{n^2}$
every matrix that is a linear
combination of the preceding matrices,
we obtain a new basis $T_1,\dots, T_r,
E_{i_1},\dots,E_{i_k}$ of the space
${\mathbb C}^{\,n\times n}$. By Lemma
\ref{t2.1}, the deformation
\[
{\cal
A}(\varepsilon_1,\dots,
\varepsilon_k)=
A+\varepsilon_1
E_{i_1}+\dots+\varepsilon_kE_{i_k}
\]
is miniversal.
\end{corollary}

For each $M\in {\mathbb C}^{m\times m}$
and $N\in {\mathbb C}^{n\times n}$,
define the vector space
\begin{equation}\label{neh1}
 T(M,N):=\{( \underbrace{S^TM+NR}_{\text{$n$-by-$m$}},\:
\underbrace{R^TN+MS}_{\text{$m$-by-$n$}})\,|\,S\in
 {\mathbb C}^{m\times n},\ R\in
 {\mathbb C}^{n\times m}\}.
\end{equation}

\begin{lemma}\label{thekd}
Let $A=A_1\oplus\dots\oplus A_t$ be a
block-diagonal matrix in which every
$A_i$ is $n_i\times n_i$. Let ${\cal
D}=[{\cal D}_{ij}]$ be a $(0,\!*)$
matrix of the same size and partitioned
into blocks conformably to the
partition of $A$. Then $A+{\cal D}(\vec
{\varepsilon})$ is a simplest
miniversal deformation of $A$ for
congruence if and only if
\begin{itemize}
  \item[\rm(i)] each affine
      subspace of ${\mathbb
      C}^{n_i\times n_i}$ parallel
      to $T(A_i)$ $($which is
      defined in
       \eqref{eelie}$)$ intersects
       ${\cal D}_{ii}(\mathbb C)$
       at exactly one point,
      and

  \item[\rm(ii)] each affine
      subspace of ${\mathbb
      C}^{n_j\times n_i}\oplus
      {\mathbb C}^{n_i\times n_j}$
      parallel to $T(A_i,A_j)$
      $($which is defined in
       \eqref{neh1}$)$
      intersects ${\cal
      D}_{ji}(\mathbb C)\oplus{\cal
      D}_{ij}(\mathbb C)$ at
      exactly one point.
\end{itemize}
\end{lemma}

\begin{proof}
By Lemma \ref{t2.1}(iii), $A+{\cal
D}(\vec {\varepsilon})$ is a simplest
miniversal deformation of $A$ if and
only if for each $C\in{\mathbb
C}^{n\times n}$ the affine subspace
$C+T(A)$ contains exactly one
$D\in{\cal D}(\mathbb C)$; that is, for
each $C$ exactly one matrix in ${\cal
D}(\mathbb C)$ has the form
\begin{equation}\label{kid}
D=C+S^TA+AS\in{\cal
D}(\mathbb C),\qquad
S\in{\mathbb
C}^{n\times n}.
\end{equation}
Let us partition $D,\ C$, and $S$ into
blocks conformably to the partition of
$A$. By \eqref{kid}, for each $i$ we
have $D_{ii}=C_{ii}+S_{ii}^TA_{i}
+A_{i}S_{ii}$, and for all $i$ and $j$
such that $i<j$ we have
\begin{equation*}\label{mht}
\begin{bmatrix}
D_{ii}&D_{ij}
 \\ D_{ji}&D_{jj}
\end{bmatrix}
=
\begin{bmatrix}
C_{ii}&C_{ij}
 \\ C_{ji}&C_{jj}
\end{bmatrix}
+ \begin{bmatrix}
S_{ii}^T&S_{ji}^T
 \\ S_{ij}^T&S_{jj}^T
\end{bmatrix}
\begin{bmatrix}
A_i&0
 \\ 0& A_j
\end{bmatrix}
+
\begin{bmatrix}
A_i&0
 \\ 0& A_j
\end{bmatrix}
\begin{bmatrix}
S_{ii}&S_{ij}
 \\ S_{ji}&S_{jj}
\end{bmatrix}.
\end{equation*}
Thus, \eqref{kid} is equivalent to the
conditions
\begin{equation}\label{djh}
D_{ii}=C_{ii}
+S_{ii}^TA_i+A_iS_{ii}\in{\cal
D}_{ii}(\mathbb
C)\quad\text{for }1\le i\le t
\end{equation}
and
\begin{equation}\label{djhh}
(D_{ji},D_{ij})= (C_{ji},C_{ij}) +(S_{ij}^TA_i+A_jS_{ji},\:
S_{ji}^TA_j+A_iS_{ij}) \in {\cal
D}_{ji}(\mathbb C)\oplus {\cal
D}_{ij}(\mathbb C)
\end{equation}
for $1\le i<j\le t$. Hence, for each
$C\in{\mathbb C}^{n\times n}$ there
exists exactly one $D\in{\cal D}$ of
the form \eqref{kid} if and only if
\begin{itemize}
  \item[(i$'$)] for each
      $C_{ii}\in{\mathbb
      C}^{n_i\times n_i}$ there
      exists exactly one
      $D_{ii}\in{\cal D}_{ii}$ of
      the form \eqref{djh}, and
  \item[(ii$'$)] for each $(
      C_{ji},C_{ij})\in{\mathbb
      C}^{n_j\times
      n_i}\oplus{\mathbb
      C}^{n_i\times n_j} $ there
      exists exactly one
      $(D_{ji},D_{ij})\in {\cal
      D}_{ji}(\mathbb C)\oplus{\cal
      D}_{ij}(\mathbb C) $ of the
      form \eqref{djhh}.
\end{itemize}
\end{proof}

\begin{corollary}\label{the}
In the notation of Lemma \ref{thekd},
$A+{\cal D}(\vec {\varepsilon})$ is a
miniversal deformation of $A$ if and
only if each submatrix of $A+{\cal
D}(\vec {\varepsilon})$ of the form
\begin{equation*}\label{a8}
\begin{bmatrix}
  A_i+{\cal D}_{ii}(\vec
{\varepsilon}) &
  {\cal D}_{ij}(\vec
{\varepsilon})\\
  {\cal D}_{ji}(\vec
{\varepsilon}) &A_j+
{\cal D}_{jj}(\vec
{\varepsilon})
\end{bmatrix}\quad \text{with }i<j
\end{equation*}
is a miniversal deformation of
$A_i\oplus A_j$. A similar reduction to
the case of canonical forms for
congruence with two direct summands was
used in \cite{t_d} for the solution of
the equation $XA + AX^T  = 0$.
\end{corollary}

We are ready to prove Theorem
\ref{teojy}. Let $A_{\rm can}=A_1\oplus
A_2\oplus\cdots\oplus A_t$ be the
canonical matrix \eqref{kus}, and let
${\cal D}=[{\cal D}_{ij}]_{i,j=1}^t$ be
the $(0,\!*)$ matrix constructed in
Theorem \ref{teojy}. Each $A_i$ has the
form $H_n(\lambda)$, or $\Gamma_n$, or
$ J_n(0)$, and so there are 3 types of
diagonal blocks ${\cal D}(A_i)={\cal
D}_{ii}$ and 6 types of pairs of
off-diagonal blocks ${\cal
D}(A_i,A_j)=({\cal D}_{ji},{\cal
D}_{ij})$, $i<j$; they were defined in
\eqref{KEV}--\eqref{lsiu6}. In the next
3 sections, we prove that all blocks of
$\cal D$ satisfy the conditions (i) and
(ii) of Lemma \ref{thekd}.

\section{Diagonal blocks of $\cal D$}

Let us verify that the diagonal blocks
of $\cal D$ defined in part (i) of
Theorem \ref{teo2} satisfy the
condition (i) of Lemma \ref{thekd}.

\subsection{Diagonal blocks ${\cal
D}(H_{n}(\lambda))$} \label{sub2}

Due to Lemma \ref{thekd}(i), it
suffices to prove that each
$2n$-by-$2n$ matrix
$A=[A_{ij}]_{i,j=1}^2$ can be reduced
to exactly one matrix of the form
\eqref{KEV} by adding
\begin{multline*}\label{moh}
\begin{bmatrix}
S_{11}^T&S_{21}^T
 \\ S_{12}^T&S_{22}^T
\end{bmatrix}
\begin{bmatrix}
0&I_n
 \\ J_n(\lambda)&0
\end{bmatrix}
+\begin{bmatrix} 0&I_n
 \\ J_n(\lambda)&0
\end{bmatrix}
\begin{bmatrix}
S_{11}&S_{12}
 \\ S_{21}&S_{22}
\end{bmatrix}
    \\=
\begin{bmatrix}
S_{21}^T J_n(\lambda)+S_{21}&
S_{11}^T+S_{22}\\
S_{22}^T J_n(\lambda)+J_n(\lambda)
S_{11}& S_{12}^T+J_n(\lambda)S_{12}
\end{bmatrix}
\end{multline*}
in which $S=[S_{ij}]_{i,j=1}^2$ is an
arbitrary $2n$-by-$2n$ matrix. Taking
$S_{22}=-A_{12}$ and the other
$S_{ij}=0$, we obtain a new matrix $A$
with $A_{12}=0$. To preserve $A_{12}$,
we hereafter must take $S$ with
$S_{11}^T+S_{22}=0$. Therefore, we can
add $S_{21}^T J_n(\lambda)+S_{21}$ to
the (new) $A_{11}$,
$S_{12}^T+J_n(\lambda)S_{12}$ to
$A_{22}$, and $-S_{11}
J_n(\lambda)+J_n(\lambda) S_{11}$ to
$A_{21}$. Using these additions, we can
reduce $A$ to the form \eqref{KEV} on
the strength of  the following 3
lemmas.

\begin{lemma}\label{lem2aaaa}
Adding $SJ_n(\lambda)+S^T$ with a fixed
$\lambda $ and an arbitrary $S$, we can
reduce each $n\times n$ matrix to
exactly one matrix of the form
\begin{equation}\label{kdi}
\begin{cases}
    0 & \text{if $\lambda\ne
\pm 1$}, \\
   0^{\nwmodels} &
\text{if $\lambda=
1$},
           \\
   0^{\nwvdash} &
\text{if $\lambda=
-1$}.
  \end{cases}
\end{equation}
\end{lemma}

\begin{proof}
Let $A=[a_{ij}]$ be an arbitrary
$n\times n$ matrix. We will reduce it
along its \emph{skew diagonals}
\[
\begin{tikzpicture}\small
\matrix (magic) [matrix of math nodes]
  {%
  a_{11} & a_{12} & a_{13} & \cdots & a_{1n} \\
  a_{21} & a_{22} & a_{23} & \cdots & a_{2n} \\
  a_{31} & a_{32} & a_{33} &\cdots & a_{3n} \\
  \cdots\  & \cdots & \cdots & \cdots & \ \cdots  \\
  a_{n1} & a_{n2} & a_{n3} & \cdots & a_{nn} \\  };
   \draw[help lines] (magic-1-1.south west) -- (magic-1-1.north east);
\draw[help lines] (magic-2-1.south
west) -- (magic-1-2.north east);
\draw[help lines] (magic-3-1.south
west) -- (magic-1-3.north east);
\draw[help lines] (magic-4-1.south
west) -- (magic-1-4.north east);
\draw[help lines] (magic-5-1.south
west) -- (magic-1-5.north east);
\draw[help lines] (magic-5-2.south
west) -- (magic-2-5.north east);
\draw[help lines] (magic-5-3.south
west) -- (magic-3-5.north east);
\draw[help lines] (magic-5-4.south
west) -- (magic-4-5.north east);
\draw[help lines] (magic-5-5.south
west) -- (magic-5-5.north east);
\end{tikzpicture}
\]
starting from the upper left corner;
that is, in the following order:
\begin{equation}\label{ly}
a_{11},\
(a_{21},a_{12}),\
(a_{31},a_{22},a_{13}),\
\dots,\ a_{nn}.
\end{equation}
We reduce $A$ by adding $\Delta
A:=SJ_n(\lambda)+S^T$ in which
$S=[s_{ij}]$ is any $n\times n$ matrix.
For instance, if $n=4$ then
\[
\Delta A=\begin{bmatrix}
 \lambda s_{11}+0+s_{11} &
 \lambda s_{12}+s_{11}+s_{21} &
 \lambda s_{13}+s_{12}+s_{31} &
 \lambda s_{14}+s_{13}+s_{41}
            \\
\lambda s_{21}+0+s_{12} &
 \lambda s_{22}+s_{21}+s_{22} &
 \lambda s_{23}+s_{22}+s_{32} &
 \lambda s_{24}+s_{23}+ s_{42}
             \\
\lambda s_{31}+0+s_{13} & \lambda
s_{32}+s_{31}+s_{23} & \lambda
s_{33}+s_{32}+s_{33} & \lambda
s_{34}+s_{33}+s_{43}
  \\
\lambda s_{41}+0+ s_{14} & \lambda
s_{42}+s_{41}+ s_{24} & \lambda
s_{43}+s_{42}+ s_{34} & \lambda
s_{44}+s_{43}+s_{44}
\end{bmatrix}.
\]
\medskip

\noindent\emph{Case 1: $\lambda\ne \pm
1$.} We reduce $A$ to $0$ by induction:
Assume that the first $t-1$ skew
diagonals of $A$ in the sequence
\eqref{ly} are zero. To preserve them,
we must and will take the first $t-1$
skew diagonals of $S$ equal to zero. If
the $t^{\text{\rm th}}$ skew diagonal
of $S$ is $(x_1,\dots,x_r)$, then we
can add
\begin{equation}\label{mugh}
(\lambda x_{1}+x_r,\
\lambda
x_{2}+x_{r-1},\
\lambda
x_{3}+x_{r-2},\ \dots,\ \lambda
x_{r}+x_1)
\end{equation}
to the $t^{\text{\rm th}}$ skew
diagonal of $A$. Each vector
$(c_1,\dots,c_r)\in\mathbb C^{r}$ is
represented in the form \eqref{mugh}
since the corresponding system of
linear equations
\[
\lambda
x_{1}+x_r=c_1,\quad
\lambda
x_{2}+x_{r-1}=c_2,\ \
\dots,\ \ \lambda
x_{r}+x_1 =c_r
\]
has a nonzero determinant for all
$\lambda\ne \pm 1$. We make the
$t^{\text{\rm th}}$ skew diagonal of
$A$ equal to zero.
\medskip

\noindent\emph{Case 2: $\lambda=1$.} We
say that a vector
$(v_1,v_2,\dots,v_r)\in {\mathbb C}^r$
is \emph{symmetric} if it is equal to
$(v_r,\dots,v_2,v_1)$, and
\emph{skew-symmetric} if it is equal to
$(-v_r,\dots,-v_2,-v_1)$. Let us
consider the equality
\begin{equation}\label{neo}
(x_1,x_2,\dots,x_r)
+(0,y_2,\dots,y_r)
=(a_1,a_2,\dots,a_r)
\end{equation}
in which $\vec x=(x_1,\dots,x_r)$ is
symmetric and $\vec y=(y_2,\dots,y_r)$
is skew-symmetric. The following two
statements hold:
\begin{itemize}
  \item[(a)] If $r$ is odd, then
      for each $a_1,\dots,a_{r}$
      there exist unique $\vec x$
      and $\vec y$ satisfying
      \eqref{neo}.
  \item[(b)] If $r$ is even, then
      for each $a_1,\dots,a_{r-1}$
      there exist unique $a_r$,
      $\vec x$, and $\vec y$
      satisfying \eqref{neo}, and
      for each $a_2,\dots,a_{r}$
      there exist unique $a_1$,
      $\vec x$, and $\vec y$
      satisfying \eqref{neo}.
\end{itemize}

Indeed, if $r=2k+1$, then \eqref{neo}
takes the form
\[
(x_1,\dots,x_k,x_{k+1},x_k,\dots,x_1)+
(0,y_2,\dots,y_{k+1},
-y_{k+1},\dots,-y_2)
=(a_1,\dots,a_{2k+1}),
\]
and so it can be rewritten as follows:
\begin{equation*}\label{mod}
\left[\begin{array}{cccc|ccc}
   1&&&&   0&&   \\
&1&&&   1&\ddots&   \\
&&\ddots&&   &\ddots&0   \\
&&&1&   &&1   \\
\hline
&&1&0   &&&-1   \\
&\udots&\udots&
&&\udots&   \\
1&0&&   &-1&   \\
\end{array}\right]
\begin{bmatrix}
x_1\\x_2\\\vdots\\x_{k+1}
\\\hline y_2\\\vdots\\y_{k+1}
\end{bmatrix}
=\begin{bmatrix}
a_1\\a_2\\\vdots\\a_{k+1}
\\\hline a_{k+2}\\\vdots
\\a_{2k+1}
\end{bmatrix}.
\end{equation*}
The matrix of this system is
nonsingular since we can add the
columns of the second vertical strip to
the corresponding columns of the first
vertical strip and reduce it to the
form
\[
\left[\begin{array}{c|c}
   \begin{matrix}
1&&&\\1&1&&\\&\ddots&\ddots&\\&&1&1
   \end{matrix}&
   \begin{matrix}
   0\ &&   \\
 1\ &\ddots&\phantom{\quad}   \\
 &\ddots&\  0   \\
 &&\  1   \end{matrix}
 \\\hline 0&\begin{matrix}
   &&-1   \\
   &\udots&   \\
   -1&   \end{matrix}
\end{array}\right]
\]
with nonsingular diagonal blocks. This
proves (a).

If $r=2k$, then \eqref{neo} takes the
form
\[
(x_1,\dots,x_k,x_k,\dots,x_1)+
(0,y_2,\dots,y_k,0,-y_k,\dots,-y_2)
=(a_1,\dots,a_{2k}),
\]
and so it can be rewritten as follows:
\begin{equation*}\label{maod}
\left[\begin{array}{cccc|ccc}
   1&&&&   0&&   \\
&1&&&   1&\ddots&   \\
&&\ddots&&   &\ddots&0   \\
&&&1&   &&1   \\
\hline
&&&1   &&&0   \\
&&\udots&   &&\udots&-1   \\
&1&&   &0&\udots&   \\
1&&&   &-1&   \\
\end{array}\right]
\begin{matrix}
\begin{bmatrix}
x_1\\x_2\\\vdots\\x_{k}
\\\hline y_2\\\vdots\\y_{k}
\end{bmatrix}\\\phantom{2}
\end{matrix}
=\begin{bmatrix}
a_1\\a_2\\\vdots\\a_{k}
\\\hline a_{k+1}\\a_{k+2}\\\vdots
\\a_{2k}
\end{bmatrix}.
\end{equation*}
The matrix of this system is
$2k$-by-$(2k-1)$ and can be reduced as
follows. For $i=1,\dots,k-1$, we add
the $i^{\text{\rm th}}$ column of the
second vertical strip to the
$i^{\text{\rm th}}$ column of the first
vertical strip or subtract it from the
$(i+1)^{\text{\rm st}}$ column of the
first vertical strip and obtain
\[
\left[\begin{array}{cccc|ccc}
   1&&&&   0&&   \\
1&1&&&   1&\ddots&   \\
&\ddots&\ddots&&   &\ddots&0   \\
&&1&1&   &&1   \\
\hline
&&&1   &&&0   \\
&&0&   &&\udots&-1   \\
&\udots&&   &0&\udots&   \\
0&&&   &-1&   \\
\end{array}\right]
\quad\text{or}\quad
\left[\begin{array}{cccc|ccc}
   1&&&&   0&&   \\
&0&&&   1&\ddots&   \\
&&\ddots&&   &\ddots&0   \\
&&&0&   &&1   \\
\hline
&&&1   &&&0   \\
&&\udots&1   &&\udots&-1   \\
&1&\udots&   &0&\udots&   \\
1&1&&   &-1&   \\
\end{array}\right].
\]
The first matrix without the first row
and the second matrix without the last
row are nonsingular. This proves (b).
\medskip

Since $\lambda=1$, we can add $\Delta
A=SJ_n(1)+S^T$ to $A$. The matrix $S$
is an arbitrary of size $n\times n$,
write it in the form $S=B+C$ in which
\begin{equation}\label{noiu}
B:=\frac{S+S^T}{2}\quad \text{and}\quad
C:=\frac{S-S^T}{2}
\end{equation}
are its symmetric and skew-symmetric
parts. Then
\[SJ_n(1)+S^T=S+SJ_n(0)+S^T
=2B+(B+C)J_n(0),\] and
so we can add to $A$ any matrix
\begin{equation}\label{kep}
\Delta
A=2B+(B+C)J_n(0)
\end{equation}
in which $B=[b_{ij}]$ is symmetric and
$C=[c_{ij}]$ is skew-symmetric.

We reduce $A$ to the form
$0^{\nwmodels}$ along the skew
diagonals \eqref{ly} as follows. Taking
$b_{11}=-a_{11}/2$, we make the $(1,1)$
entry of $A$ equal to zero. Reasoning
by induction, we fix
$t\in\{1,\dots,n-1\}$ and assume that
\begin{itemize}
  \item the first $t-1$ skew
      diagonals of $A$ have been
      reduced to the form
      $0^{\nwmodels}$ (that is,
      these diagonals coincide with
      the corresponding skew
      diagonals of some matrix of
      the form $0^{\nwmodels}$) and
these skew diagonals are uniquely
determined by the initial matrix
$A$;

  \item if $t\le n$ and $S$
      preserves the first $t-1$
      skew diagonals of $A$ (i.e.,
      the first $t-1$ skew
      diagonals of \eqref{kep} are
      zero) then the first $t-1$
      skew diagonals of $B$ are
      zero.
\end{itemize}

Let $t\le n$. Then the $t^{\text{\rm
th}}$ skew diagonal of \eqref{kep} has
the form
\begin{equation}\label{lwd}
(b_1,b_2,\dots,b_r)
+(0,c_2,\dots,c_r)
\end{equation}
in which $(b_1,b_2,\dots,b_r)$ is an
arbitrary symmetric vector (it is the
$t^{\text{\rm th}}$ skew diagonal of
$2B$) and $(c_2,c_3,\dots,c_r)$ is an
arbitrary skew-symmetric vector (it is
the $(t-1)^{\text{\rm st}}$ skew
diagonal of $C$). The statements (a)
and (b) imply that we can make the
$t^{\text{\rm th}}$ skew diagonal of
$A$ as in $0^{\nwmodels}$ by adding
\eqref{lwd}. Moreover, this skew
diagonal is uniquely determined, and to
preserve it the $t^{\text{\rm th}}$
skew diagonal of $B$ must be zero.

For instance, if $t=2\le n$, then we
add $(b_1,b_1)$ and reduce the second
skew diagonal of $A$ to the form
$(*,0)$ or $(0,*)$. If $t=3\le n$, then
we add $(b_1,b_2,b_1)+(0,c_2,-c_2)$ and
make the third skew diagonal of $A$
equal to $0$.

Let $t>n$. Let us take $S$ in which the
first $t-1$ skew diagonals are equal to
$0$. Then the $t^{\text{\rm th}}$ skew
diagonal of \eqref{kep} has the form
\begin{equation}\label{lwed}
(b_1,b_2,\dots,b_r)
+(c_1,c_2,\dots,c_r)
\end{equation}
in which $(b_1,b_2,\dots,b_r)$ is the
$t^{\text{\rm th}}$ skew diagonal of
$2B$ is symmetric and
$(c_1,c_2,\dots,c_r)$ is the
$(t-1)^{\text{\rm st}}$ skew diagonal
of $C$ without the last entry. Thus,
$(b_1,b_2,\dots,b_r)$ is any symmetric,
$c_1$ is arbitrary, and
$(c_2,c_3,\dots,c_r)$ is any
skew-symmetric. Adding \eqref{lwed}, we
reduce the $t^{\text{\rm th}}$ skew
diagonal of $A$ to $0$.
\medskip

\noindent\emph{Case 3: $\lambda=-1$.}
We can add $SJ_n(-1)+S^T$ to $A$. Write
$S$ in the form $B+C$, in which $B$ and
$C$ are defined in \eqref{noiu}. Then
\[
\Delta A=
SJ_n(-1)+S^T=-S+SJ_n(0)+S^T
=-2C+(B+C)J_n(0).
\]

We reduce $A$ to the form
$0^{\nwvdash}$ along the skew diagonals
\eqref{ly} as follows. The $(1,1)$
entry of $\Delta A$ is $0$; so we
cannot change $a_{11}$. Reasoning by
induction, we fix $t\in\{1,\dots,n-1\}$
and assume that
\begin{itemize}
  \item the first $t-1$ skew
      diagonals of $A$ have been
      reduced to the form
      $0^{\nwvdash}$ and these
      diagonals are uniquely
      determined by the initial
      matrix $A$;
    \item if $t\le n$ and $S$
        preserves the first $t-1$
        skew diagonals of $A$
        (i.e., the first $t-1$ skew
        diagonals of \eqref{kep}
        are zero) then the first
        $t-1$ skew diagonals of $C$
        are zero.
\end{itemize}

If $t\le n$, then we can add to the
$t^{\text{\rm th}}$ skew diagonal of
$A$ any vector
\[
(c_1,c_2,\dots,c_t)+(0,b_2,\dots,b_t)
\]
in which $(c_1,\dots,c_t)$ is
skew-symmetric (it is the $t^{\text{\rm
th}}$ skew diagonal of $-2C$)  and
$(b_2,\dots,b_t)$ is symmetric (it is
the $(t-1)^{\text{\rm st}}$ skew
diagonal of $B$). We make the
$t^{\text{\rm th}}$ skew diagonal of
$A$ as in $0^{\nwvdash}$. For instance,
if $t=2\le n$, then we add
$(c_1,-c_1)+(0,b_2)$ and make the
second skew diagonal of $A$ equal to
zero. If $t=3\le n$, then we add
$(c_1,0,-c_1)+(0,b_2,b_2)$ and reduce
the third skew diagonal of $A$ to the
form $(*,0,0)$ or $(0,0,*)$.

Let $t>n$. Let us take $S$ in which the
first $t-1$ skew diagonals are equal to
$0$. Then we can add to the
$t^{\text{\rm th}}$ skew diagonal of
$A$ any vector
\[
(c_1,c_2,\dots,c_r)+(b_1,b_2,\dots,b_r)
\]
in which $(c_1,\dots,c_r)$ is
skew-symmetric (it is the $t^{\text{\rm
th}}$ skew diagonal of $-2C$), $b_1$ is
arbitrary, and $(b_2,\dots,b_{r-1})$ is
symmetric (it is the $(t-1)^{\text{\rm
st}}$ skew diagonal of $B$ without the
first and the last elements). We make
the $t^{\text{\rm th}}$ skew diagonal
of $A$ equal to zero.
\end{proof}

\begin{lemma}\label{lem2a}
Adding $J_n(\lambda)S+S^T$, we can
reduce each $n\times n$ matrix to
exactly one matrix of the form
\begin{equation}\label{kdis}
\begin{cases}
    0 & \text{if $\lambda\ne
\pm 1$},
             \\
   0^{\semodels}
& \text{if $\lambda=
1$},
           \\
   0^{\sevdash}
& \text{if $\lambda=
-1$}.
  \end{cases}
\end{equation}
\end{lemma}

\begin{proof}
By Lemma \ref{lem2aaaa}, for each
$n\times n$ matrix $B$ there exists $R$
such that $M:=B+RJ_n(\lambda)+R^T$ has
the form \eqref{kdi}. Then
\begin{equation*}
M^T=B^T+J_n(\lambda)^TR^T+R.
\end{equation*}
Write
\[
Z:= \begin{bmatrix}
 0&&1
 \\ &\udots&\\
 1&&0
 \end{bmatrix}.
\]
Because
$ZJ_n(\lambda)^TZ=J_n(\lambda)$, we
have
\[
ZM^TZ=ZB^TZ+J_n(\lambda)(ZRZ)^T+ZRZ.
\]
This ensures Lemma \ref{lem2a} since
$ZB^TZ$ is arbitrary and $ZM^TZ$ is of
the form \eqref{kdis}.
\end{proof}

\begin{lemma}\label{lem3a}
Adding $S J_n(\lambda)-J_n(\lambda) S$,
we can reduce each $n\times n$ matrix
to exactly one matrix of the form
$0^{\swarrow}$.
\end{lemma}

\begin{proof}
Let $A=[a_{ij}]$ be an arbitrary
$n\times n$ matrix. Adding
\begin{multline*}
SJ_n(\lambda)-J_n(\lambda)S
=SJ_n(0)-J_n(0)S \\=
\begin{bmatrix}
 s_{21}-0&s_{22}-s_{11}&s_{23}-s_{12}
 &\dots&s_{2n}-s_{1,n-1}
 \\ \hdotsfor{5}\\
 s_{n-1,1}-0&s_{n-1,2}-s_{n-2,1}
 &s_{n-1,3}-s_{n-2,2}
 &\dots&s_{n-1,n}-s_{n-2,n-1}\\
 s_{n1}-0&s_{n2}-s_{n-1,1}&s_{n3}-s_{n-1,2}
 &\dots&s_{nn}-s_{n-1,n-1}
 \\ 0-0&0-s_{n1}&0-s_{n2}&\dots&0-s_{n,n-1}
 \end{bmatrix},
\end{multline*}
we reduce $A$ along the diagonals
\begin{equation*}\label{lyj}
a_{n1},\
(a_{n-1,1},a_{n2}),\
(a_{n-2,1},a_{n-1,2},
a_{n3}),\ \dots,\ a_{1n}
\end{equation*}
to the form $0^{\swarrow}$.
\end{proof}

\subsection{Diagonal blocks ${\cal D}
(\Gamma_n)$} \label{sub6}

Due to Lemma \ref{thekd}(i), it
suffices to prove that each $n\times n$
matrix $A$ can be reduced to exactly
one matrix of the form \eqref{uwm} by
adding $\Delta
A:=S^T\Gamma_n+\Gamma_nS$. Write
$\Gamma_n$ as the sum of its symmetric
and skew-symmetric parts:
$\Gamma_n=\Gamma_n^{(s)}+\Gamma_n^{(c)},$
where
\[
\Gamma_n^{(s)}=
    \begin{bmatrix}
0&&&&\udots
\\&&&0&
\udots
\\
&&0&1&\\ &0&-1&\\ 0&1&
&&0
\end{bmatrix}
     \ \text{ and }\
\Gamma_n^{(c)}=
\begin{bmatrix}
0&&&&\udots
\\&&&-1&\\
&&1&&\\ &-1&&\\ 1&&
&&0
\end{bmatrix}\
\text{if $n$ is even},
\]
\[
\Gamma_n^{(s)}=
\begin{bmatrix}
0&&&&\udots
\\&&&-1&\\
&&1&&\\ &-1&&\\ 1&&
&&0
\end{bmatrix}
     \ \text{ and }\
\Gamma_n^{(c)}=
    \begin{bmatrix}
0&&&&\udots
\\&&&0&
\udots\\
&&0&1&\\ &0&-1&\\ 0&1&
&&0
\end{bmatrix}
     \
\text{if $n$ is odd}.
\]
Then the symmetric and skew-symmetric
parts of $\Delta A$ are
\[
\Delta
A^{(s)}=S^T\Gamma_n^{(s)}
+\Gamma_n^{(s)}S,\qquad
\Delta
A^{(c)}=S^T\Gamma_n^{(c)}
+\Gamma_n^{(c)}S
\]
in which $S=[s_{ij}]$ is any $n\times
n$ matrix.
\medskip

\noindent\emph{Case 1: $n$ is even.}
Then
\[
\Delta
A^{(s)}=\begin{bmatrix}
0+0&s_{n1}+0&-s_{n-1,1}+0&\dots&s_{21}+0\\
0+s_{n1}&s_{n2}+s_{n2}&-s_{n-1,2}+s_{n3}
&\dots&s_{22}+s_{nn}\\
0-s_{n-1,1}&s_{n3}-s_{n-1,2}
&-s_{n-1,3}-s_{n-1,3}
&\dots&s_{23}-s_{n-1,n}\\
\hdotsfor{5}\\
0+s_{21}&s_{nn}+s_{22}&-s_{n-1,n}+s_{23}
&\dots&s_{2n}+s_{2n}\\
\end{bmatrix},
\]
\[
\Delta A^{(c)}= \begin{bmatrix}
s_{n1}-s_{n1}&-s_{n-1,1}-s_{n2}
&s_{n-2,1}-s_{n3}
&\dots&-s_{11}-s_{nn}\\
s_{n2}+s_{n-1,1}&-s_{n-1,2}+s_{n-1,2}
&s_{n-2,2}+s_{n-1,3}
&\dots&-s_{12}+s_{n-1,n}\\
s_{n3}-s_{n-2,1}&-s_{n-1,3}-s_{n-2,2}
&s_{n-2,3}-s_{n-2,3}
&\dots&-s_{13}-s_{n-2,n}\\
\hdotsfor{5}\\
s_{nn}+s_{11}&-s_{n-1,n}+s_{12}
&s_{n-2,n}+s_{13}
&\dots&-s_{1n}+s_{1n}\\
\end{bmatrix}.
\]

We reduce $A=[a_{ij}]$ to the form
$0^{\nwvdash}$ along its skew diagonals
\eqref{ly} as follows. The $(1,1)$
entry of $\Delta A=\Delta
A^{(s)}+\Delta A^{(c)}$ is zero, and so
the $(1,1)$ entry of $A$ is not
changed. Reasoning by induction, we fix
$t\in\{1,\dots,n-1\}$ and assume that
\begin{itemize}
  \item the first $t$ skew
      diagonals of $A$ have been
      reduced to the form
      $0^{\nwvdash}$ and they are
      uniquely determined by the
      initial $A$;

  \item the addition of $\Delta A$
      preserves the first $t$ skew
      diagonals of $A$ if and only
      if the first $t-1$ diagonals
 of $S$ starting from the lower
 left diagonal
\[
\begin{tikzpicture}\small
\matrix (magic) [matrix of math nodes]
  {%
  \cdots & \cdots & \cdots & \cdots  & \cdots \\
  s_{n-t+2,1}\!\!\!\! & \cdots & \cdots& \cdots & \cdots \\
    \cdots\  & \cdots & \cdots & \cdots & \ \cdots  \\
s_{n-1,1} & \cdots &  \cdots& \cdots  & \cdots \\
  \ s_{n1}\  & s_{n2}& \cdots & s_{n,t-1}
   & \cdots \\  };
\draw[help lines] (magic-5-4.south
east) -- (magic-2-1.north west);
\draw[help lines] (magic-5-2.south
east) -- (magic-4-1.north west);
\draw[help lines] (magic-5-1.south
east) -- (magic-5-1.north west);
\end{tikzpicture}
\]
are zero and its $t^{\text{th}}$
      diagonal
      \[ (s_{n-t+1,1},s_{n-t+2,2},
      s_{n-t+3,3}\dots,s_{n-2,t-2},
      s_{n-1,t-1},s_{n,t})
      \] is symmetric if $t$
      is odd and skew-symmetric if
      $t$ is even.
\end{itemize}

Write
\[
(v_1,v_2,\dots,v_t):=\left(
(-1)^{t-1}s_{n-t+1,1},
      \dots,s_{n-2,t-2},
      -s_{n-1,t-1},s_{n,t}
\right);
\]
this vector is symmetric for all $t$.

The $(t+1)^{\text{st}}$ skew diagonal
of $\Delta A^{(s)}$ is
\begin{equation}\label{jy}
(0,v_t,v_{t-1},\dots,v_2,v_1)
+(v_1,v_2,\dots,v_{t-1},v_t,0).
\end{equation}
If $t$ is odd, then every symmetric
vector of dimension $t+1$ is
represented in the form \eqref{jy}. If
$t$ is even, then \eqref{jy} without
the first and the last elements is an
arbitrary symmetric vector of dimension
$t-1$ and the first (and the last)
element of \eqref{jy} is fully
determined by the other elements. Since
the $(t+1)^{\text{st}}$ skew diagonal
of $\Delta A^{(c)}$ is an arbitrary
skew-symmetric vector of dimension
$t+1$, this means that the
$(t+1)^{\text{st}}$ skew diagonal of
$A$ is reduced to zero if $t$ is odd,
and to the form $(*,0,\dots,0)$ or
$(0,\dots,0,*)$ if $t$ is even. To
preserve it, we hereafter must take
those $S$ in which the
$(t+1)^{\text{st}}$ skew diagonal of
$\Delta A^{(c)}$ is zero; this means
that the $(t+1)^{\text{st}}$ diagonal
of $S$ is symmetric if $t+1$ is odd and
skew-symmetric if $t+1$ is even.

Thus, the first $n$ skew diagonals in
$A$ have the form of the corresponding
diagonals in $0^{\nwvdash}$.

The $(n+1)^{\text{st}}$ skew diagonal
of $\Delta A^{(s)}$ has the form
\[
(v_n,v_{n-1},\dots,v_2)
+(v_2,\dots,v_{n-1},v_n),
\]
(compare with \eqref{jy}) and every
symmetric vector of dimension $n-1$ is
represented in this form. Hence, the
$(n+1)^{\text{st}}$ skew diagonal of
$\Delta A$ is an arbitrary vector of
dimension $n-1$ and we make the
$(n+1)^{\text{st}}$ skew diagonal of
$A$  equal to zero. Analogously, we
make its $n+2$, $n+3,\ \dots$ skew
diagonals equal to zero and reduce $A$
to the form $0^{\nwvdash}$.
\medskip

\noindent\emph{Case 2: $n$ is odd.}
Then
\[
\Delta A^{(s)}=
\begin{bmatrix}
s_{n1}+s_{n1}&-s_{n-1,1}+s_{n2}
&s_{n-2,1}+s_{n3}
&\dots&s_{11}+s_{nn}\\
s_{n2}-s_{n-1,1}&-s_{n-1,2}-s_{n-1,2}
&s_{n-2,2}-s_{n-1,3}
&\dots&s_{1n}-s_{n-1,n}\\
s_{n3}+s_{n-2,1}&-s_{n-1,3}+s_{n-2,2}
&s_{n-2,3}+s_{n-2,3}
&\dots&s_{13}+s_{n-2,n}\\
\hdotsfor{5}\\
s_{nn}+s_{11}&-s_{n-1,n}+s_{12}
&s_{n-2,n}+s_{13}
&\dots&s_{1n}+s_{1n}\\
\end{bmatrix},
\]
\[
\Delta
A^{(c)}=\begin{bmatrix}
0+0&s_{n1}+0&-s_{n-1,1}+0&\dots&-s_{21}+0\\
0-s_{n1}&s_{n2}-s_{n2}&-s_{n-1,2}-s_{n3}
&\dots&-s_{22}-s_{nn}\\
0+s_{n-1,1}&s_{n3}+s_{n-1,2}
&-s_{n-1,3}+s_{n-1,3}
&\dots&-s_{23}+s_{n-1,n}\\
\hdotsfor{5}\\
0+s_{21}&s_{nn}+s_{22}&-s_{n-1,n}+s_{23}
&\dots&-s_{2n}+s_{2n}\\
\end{bmatrix}.
\]
We reduce $A$ along its skew diagonals
\eqref{ly}. The first skew diagonal of
$\Delta A^{(s)}$ is arbitrary; we make
the first entry of $A$ equal to zero.

Let $t<n$. Assume that
\begin{itemize}
  \item the first $t$ skew
      diagonals of $A$ have been
      reduced to the form
      $0^{\nwmodels}$ and they are
      uniquely determined by the
      initial $A$;
  \item the addition of $\Delta A$
      preserves these diagonals if
      and only if the first $t-1$
      diagonals of $S$, starting
      from the lower left diagonal,
      are zero and the
      $t^{\text{th}}$ diagonal
      $(u_1,\dots,u_t)$ of $S$ is
      symmetric if $t$ is even and
      skew-symmetric if $t$ is odd.

\end{itemize}
Then the vector
\[
(v_1,v_2,\dots,v_t):=\left((-1)^{t-1}u_1,
,\dots,u_{t-2},-u_{t-1},
u_t\right)
\]
is skew-symmetric for all $t$.

The $(t+1)^{\text{st}}$ skew diagonal
of $\Delta A^{(c)}$ is
\begin{equation}\label{jys}
(0,v_t,v_{t-1},\dots,v_2,v_1)
-(v_1,v_2,\dots,v_{t-1},v_t,0).
\end{equation}
If $t$ is even, then every
skew-symmetric vector of dimension
$t+1$ is represented in the form
\eqref{jys}. If $t$ is odd, then
\eqref{jys} without the first and the
last elements is an arbitrary
skew-symmetric vector of dimension
$t-1$ and the first (and the last)
element of \eqref{jys} is fully
determined by the other elements. Since
the $(t+1)^{\text{st}}$ skew diagonal
of $\Delta A^{(s)}$ is an arbitrary
symmetric vector of dimension $t+1$,
this means that the $(t+1)^{\text{st}}$
skew diagonal of $A$ reduces to $0$ if
$t$ is even, and to the form
$(*,0,\dots,0)$ or $(0,\dots,0,*)$ if
$t$ is odd. To preserve it, we
hereafter must take those $S$ in which
the $(t+1)^{\text{st}}$ skew diagonal
of $\Delta A^{(s)}$ is zero; this means
that the $(t+1)^{\text{st}}$ diagonal
of $S$ is symmetric if $t+1$ is even
and skew-symmetric if $t+1$ is odd.

The first $n$ skew diagonals in $A$
have the form of the corresponding
diagonals in $0^{\nwmodels}$. The
$(n+1)^{\text{st}}$ skew diagonal in
$\Delta A^{(c)}$ has the form
\[
(v_n,v_{n-1},\dots,v_2)
-(v_2,\dots,v_{n-1},v_n)
\]
(compare with \eqref{jys}) and every
skew-symmetric vector is represented in
this form. Hence, the
$(n+1)^{\text{st}}$ skew diagonal of
$\Delta A$ is an arbitrary vector of
dimension $n-1$ and we make the
$(n+1)^{\text{st}}$ skew diagonal of
$A$ equal to zero. Analogously, we make
its $n+2$, $n+3, \dots$ skew diagonals
equal to zero and reduce $A$ to the
form $0^{\nwmodels}$.

\subsection{Diagonal blocks ${\cal
D}(J_n(0))$} \label{sub1}

Due to Lemma \ref{thekd}(i), it
suffices to prove that each $n\times n$
matrix $A$ can be reduced to exactly
one matrix of the form \eqref{lsiu} by
adding
\begin{align}\nonumber
\Delta A:&=S^T J_n(0)+J_n(0)S
    \\\label{ed2}&=
\begin{bmatrix}
0+s_{21}&s_{11}+s_{22}&s_{21}+s_{23}
&\dots& s_{n-1,1}+s_{2n}
  \\
0+s_{31}&s_{12}+s_{32}&s_{22}+s_{33}
&\dots& s_{n-1,2}+s_{3n}
  \\
\hdotsfor{5}
  \\
0+s_{n1}&s_{1,n-1}+s_{n2}&s_{2,n-1}+s_{n3}
&\dots& s_{n-1,n-1}+s_{nn}
  \\
0+0&s_{1n}+0&s_{2n}+0 &\dots&
s_{n-1,n}+0
\end{bmatrix}
\end{align}
in which $S=[s_{ij}]$ is any $n\times
n$ matrix. Thus,
\begin{equation*}\label{uyhk}
\Delta
A=[b_{ij}],\qquad
b_{ij}:=s_{j-1,i}+s_{i+1,j}
\qquad
(s_{0i}:=0,\quad
s_{n+1,j}:=0),
\end{equation*}
and so all entries of $\Delta A$ have
the form $s_{kl}+s_{l+1,k+1}$. The
transitive closure of
$(k,l)\sim(l+1,k+1)$ is an equivalence
relation on the set
$\{1,\dots,n\}\times \{1,\dots,n\}$.
Represent $\Delta A$ as the sum
\[
\Delta
A=B_{n1}+B_{n-1,1}+\dots+B_{11}+B_{12}
+\dots+B_{1n}
\]
of matrices that correspond to the
equivalence classes and are defined as
follows. Each $B_{1j}$
$(j=1,2,\dots,n)$ is obtained from
$\Delta A$ by replacing with $0$ all of
its entries except for
\begin{equation}\label{hter}
s_{1j}+s_{j+1,2},\
s_{j+1,2}+s_{3,j+2},\
s_{3,j+2}+s_{j+3,4},\ \dots
\end{equation}
and each $B_{i1}$ $(i=2,3,\dots,n)$ is
obtained from $\Delta A$ by replacing
with $0$ all of its entries except for
\begin{equation}\label{hter1}
0+s_{i1},\
s_{i1}+s_{2,i+1}, \
s_{2,i+1}+s_{i+2,3}, \
s_{i+2,3}+s_{4,i+3}, \
s_{4,i+3}+s_{i+4,5},\ \dots;
\end{equation}
the pairs of indices in \eqref{hter}
and in \eqref{hter1} are equivalent:
\[
(1,j)\sim(j+1,2)\sim
(3,j+2)\sim(j+3,4)\sim\dots
\]
and
\[
(i,1)\sim(2,i+1) \sim
(i+2,3)\sim
(4,i+3)\sim
(i+4,5)\sim\dots .
\]
We call the entries \eqref{hter} and
\eqref{hter1} the \emph{main entries}
of $B_{1j}$ and $B_{i1}$ ($i>1$). The
matrices $B_{n1},
\dots,B_{11},B_{12},\dots, B_{1n}$ have
no common $s_{ij}$.

An arbitrary sequence of complex
numbers can be represented in the form
\eqref{hter}. The entries \eqref{hter1}
are linearly dependent only if the last
entry in this sequence has the form
$s_{kn}+0$ (see \eqref{ed2}); then
$(k,n)=(2p,i-1+2p)$ for some $p$, and
so $i=n+1-2p$. Thus the following
sequences \eqref{hter1} are linearly
dependent:
\begin{gather*}
0+s_{n-1,1},\ s_{n-1,1}+s_{2n}, \
s_{2n}+0;
   \\
0+s_{n-3,1},\ s_{n-3,1}+s_{2,n-2}, \
s_{2,n-2}+s_{n-1,3}, \
s_{n-1,3}+s_{4n}, \ s_{4n}+0;\ \dots
\end{gather*}
One of the main entries of each of the
matrices $B_{n-1,1}$, $B_{n-3,1}$,
$B_{n-5,1},\ \dots$ is the linear
combination of the other main entries
of this matrix, which are arbitrary.
The main entries of the other matrices
$B_{i1}$ and $B_{1j}$ are arbitrary.
Adding $B_{i1}$ and $B_{1j}$, we reduce
$A$ to the form $0^{\swvdash}$.

\section{Off-diagonal blocks of $\cal
D$ that correspond to summands of
$A_{\text{can}}$ of the same type}

Now we verify the condition (ii) of
Lemma \ref{thekd} for those
off-diagonal blocks of $\cal D$
(defined in Theorem \ref{teo2}(ii))
whose horizontal and vertical strips
contain summands of $A_{\text{can}}$ of
the same type.

\subsection{Pairs of blocks ${\cal
D}(H_m(\lambda),\, H_n(\mu))$}
\label{sub3}

Due to Lemma \ref{thekd}(ii), it
suffices to prove that each pair
$(B,A)$ of $2n\times 2m$ and $2m\times
2n$ matrices can be reduced to exactly
one pair of the form \eqref{lsiu1} by
adding
\[
(S^T
H_m(\lambda)+ H_n(\mu)
R,\:R^TH_n(\mu)
+H_m(\lambda)S),\quad S\in
 {\mathbb C}^{2m\times 2n},\ R\in
 {\mathbb C}^{2n\times 2m}.
\]

Taking $R=0$ and
$S=-H_m(\lambda)^{-1}A$, we reduce $A$
to $0$. To preserve $A=0$ we hereafter
must take $S$ and $R$ such that
$R^TH_n(\mu) +H_m(\lambda)S=0$; that
is,
\[
S=-H_m(\lambda)^{-1}
R^TH_n(\mu),
\]
and so we can add
\[
\Delta B:= -H_n(\mu)^T
R
H_m(\lambda)^{-T}H_m(\lambda)+
H_n(\mu) R
\]
to $B$.

 Write $
P:=-H_n(\mu)^T R,$ then
$R=-H_n(\mu)^{-T} P$ and
\begin{equation}\label{kif}
\Delta B=
 P
\begin{bmatrix}
J_m(\lambda)&0\\
0&J_m(\lambda)^{-T}
\end{bmatrix}
- \begin{bmatrix}
J_n(\mu)^{-T}&0\\0&J_n(\mu)
\end{bmatrix} P.
\end{equation}

Let us partition $B$, $\Delta B$, and
$P$ into $n\times m$ blocks:
\[
B=\begin{bmatrix}
B_{11}&B_{12}\\
B_{21}& B_{22}
\end{bmatrix},\qquad
\Delta
B=\begin{bmatrix}
\Delta B_{11}&\Delta
B_{12}\\\Delta
B_{21}&\Delta B_{22}
\end{bmatrix},\qquad
P=
\begin{bmatrix}
X&Y\\Z&T
\end{bmatrix}.
\]
By \eqref{kif},
\begin{align*}
\Delta
B_{11}&=XJ_m(\lambda)
-J_n(\mu)^{-T}X,
  &
\Delta
B_{12}&=YJ_m(\lambda)^{-T}
-J_n(\mu)^{-T}Y,
 \\
\Delta B_{21}&=
ZJ_m(\lambda)-J_n(\mu)Z,
 &
\Delta B_{22}&=
TJ_m(\lambda)^{-T}-J_n(\mu)T.
\end{align*}

These equalities show that we can
reduce each block $B_{ij}$ separately
by adding $\Delta B_{ij}$.
\medskip

(i) Fist we reduce $B_{11}$ by adding
$\Delta
B_{11}=XJ_m(\lambda)-J_n(\mu)^{-T}X$.

If $\lambda\ne \mu^{-1}$, then $\Delta
B_{11}$ is an arbitrary $n\times m$
matrix since $J_m(\lambda)$ and
$J_n(\mu)^{-T}$ have no common
eigenvalues; we make $B_{11}=0$.

Let $\lambda= \mu^{-1}$. Then
\begin{equation*}
J_n(\mu)^{-T}=
\left[\begin{MAT}(e){ccccc}
\mu^{-1}&&&&0 \\
-\mu^{-2}&\mu^{-1}&&&\\
\mu^{-3}&-\mu^{-2}&\mu^{-1}&&\\
\ddots&\ddots&\ddots&\ddots&\\
\ddots&\ddots&\mu^{-3}&-\mu^{-2}&\mu^{-1}\\
\end{MAT}\right]=
\left[\begin{MAT}(e){ccccc}
\lambda&&&&0 \\
-\lambda^{2}&\lambda&&&\\
\lambda^{3}&-\lambda^{2}&\lambda&&\\
\ddots&\ddots&\ddots&\ddots&\\
\ddots&\ddots&\lambda^{3}&-\lambda^{2}&\lambda\\
\end{MAT}\right].
\end{equation*}
Adding
\begin{multline*}
\Delta
B_{11}=XJ_m(0)-(J_n(\mu)^{-T}-\lambda
I_n)X
  \\=
\begin{bmatrix}
0&x_{11}&\dots&x_{1,m-1} \\
0&x_{21}&\dots&x_{2,m-1} \\
0&x_{31}&\dots&x_{3,m-1} \\
\hdotsfor{4}
\end{bmatrix}
    +\lambda^2
\begin{bmatrix}
0&\dots&0 \\
x_{11}&\dots&x_{1m} \\
x_{21}&\dots&x_{2m} \\
\hdotsfor{3}\\
\end{bmatrix}
    -\lambda^3
\begin{bmatrix}
0&\dots&0 \\
0&\dots&0 \\
x_{11}&\dots&x_{1m} \\
\hdotsfor{3}\\
\end{bmatrix}+\cdots,
\end{multline*}
we reduce $B_{11}$ to the form
$0^{\nwarrow}$ along the skew diagonals
starting from the upper left corner.
\medskip

(ii) Let us reduce $B_{12}$ by adding
$\Delta B_{12}
=YJ_m(\lambda)^{-T}-J_n(\mu)^{-T}Y$.

If $\lambda\ne \mu$, then $\Delta
B_{12}$ is arbitrary; we make
$B_{12}=0$.

Let $\lambda= \mu$. Write $F:=J_n(0)$.
Since
\[
J_n(\lambda )^{-1}=(\lambda I_n+F)^{-1}=
\lambda ^{-1}I_n-\lambda ^{-2}F+\lambda ^{-3}F^2-
\cdots,
\]
we have
\begin{align*}
\Delta B_{12}
&=Y(J_m(\lambda)^{-T}-\lambda^{-1} I_m)
-(J_n(\lambda)^{-T}-\lambda^{-1} I_n)Y
  \\&=
-\lambda^{-2}\begin{bmatrix}
y_{12}&\dots&y_{1m}&0 \\
y_{22}&\dots&y_{2m}&0 \\
y_{32}&\dots&y_{3m}&0 \\
\hdotsfor{4}
\end{bmatrix}
    +\lambda^{-2}
\begin{bmatrix}
0&\dots&0 \\
y_{11}&\dots&y_{1m} \\
y_{21}&\dots&y_{2m} \\
\hdotsfor{3}\\
\end{bmatrix}
+\cdots
\end{align*}
We reduce $B_{12}$ to the form
$0^{\nearrow}$ along its diagonals
starting from the upper right corner.
\medskip

(iii) Let us reduce $B_{21}$ by adding
$\Delta B_{21} =
ZJ_m(\lambda)-J_n(\mu)Z$.

If $\lambda\ne \mu$, then $\Delta
B_{21}$ is arbitrary; we make
$B_{21}=0$.

If $\lambda= \mu$, then
\begin{align*}
\Delta B_{21} &=Z(J_m(\lambda)-\lambda
I_m) -(J_n(\lambda)-\lambda I_n)Z
  \\&=
\begin{bmatrix}
0&z_{11}&\dots&z_{1,m-1} \\
\hdotsfor{4}\\
0&z_{n-1,1}&\dots&z_{n-1,m-1} \\
0&z_{n1}&\dots&z_{n,m-1}
\end{bmatrix}
    -
\begin{bmatrix}
z_{21}&\dots&z_{2m} \\
\hdotsfor{3}\\
z_{n1}&\dots&z_{nm} \\
0&\dots&0
\end{bmatrix};
\end{align*}
we reduce $B_{12}$ to the form
$0^{\swarrow}$ along its diagonals
starting from the lover left corner.
\medskip

(iv) Finally, reduce $B_{22}$ by adding
$\Delta B_{22} =
TJ_m(\lambda)^{-T}-J_n(\mu)T$.

If $\lambda\ne \mu^{-1}$, then $\Delta
B_{22}$ is arbitrary; we make
$B_{22}=0$.

If $\lambda= \mu^{-1}$, then
\begin{multline*}
\Delta B_{22} = T(J_m(\lambda)^{-T}-\mu
I_m)-(J_n(\mu)-\mu I_n)T
  \\=
-\mu^2\begin{bmatrix}
\dots&t_{1m}&0 \\
\hdotsfor{3}\\
\dots&t_{n-1,m}&0 \\
\dots&t_{nm}&0
\end{bmatrix}
        +
\mu^3\begin{bmatrix}
\dots&t_{1m}&0&0 \\
\hdotsfor{4}\\
\dots&t_{n-1,m}&0&0 \\
\dots&t_{nm}&0&0
\end{bmatrix}
-\dots
    -
\begin{bmatrix}
t_{21}&\dots&t_{2m} \\
\hdotsfor{3}\\
t_{n1}&\dots&t_{nm} \\
0&\dots&0
\end{bmatrix};
\end{multline*}
we reduce $B_{22}$ to the form
$0^{\searrow}$ along its skew diagonals
starting from the lover right corner.

\subsection{Pairs of blocks ${\cal D}
(\Gamma_m,\Gamma_n)$} \label{sub7}

Due to Lemma \ref{thekd}(ii), it
suffices to prove that each pair
$(B,A)$ of $n\times m$ and $m\times n$
matrices can be reduced to exactly one
pair of the form \eqref{lsiu2} by
adding
\[
(S^T
\Gamma_m+ \Gamma_n
R,\:R^T\Gamma_n
+\Gamma_mS),\qquad S\in
 {\mathbb C}^{m\times n},\ R\in
 {\mathbb C}^{n\times m}.
\]

Taking $R=0$ and $S=-\Gamma_m^{-1}A$,
we reduce $A$ to $0$. To preserve $A=0$
we hereafter must take $S$ and $R$ such
that $R^T\Gamma_n +\Gamma_mS=0$; that
is, $ S=-\Gamma_m^{-1} R^T\Gamma_n, $
and so we can add
\[
 \Delta B:=
-\Gamma_n^T R
\Gamma_m^{-T}\Gamma_m+\Gamma_n R
\]
to $B$.

Write $P:=\Gamma_n^T R$, then
\begin{equation*}\label{due}
\Delta B=  -P
(\Gamma_m^{-T}\Gamma_m)+(\Gamma_n
\Gamma_n^{-T})P.
\end{equation*}
Since
\begin{equation*}\label{1n}
\Gamma_n=
\begin{bmatrix} 0&&&&
\udots
\\&&&-1&\udots
\\&&1&1\\ &-1&-1& &\\
1&1&&&0
\end{bmatrix},\quad
\Gamma_n^{-1}=(-1)^{n+1}
 \begin{bmatrix}
\vdots&\vdots&\vdots&\vdots&\udots
\\
-1&-1&-1&-1&\\ 1&1&1&&\\ -1&-1&&&\\
1&&&&0
\end{bmatrix},
\end{equation*}
we have
\begin{equation*}\label{1x11}
\Gamma_m^{-T}\Gamma_m=
(-1)^{m+1}
\begin{bmatrix} 1&2&&*
\\&1&\ddots&\\
&&\ddots&2\\ 0 &&&1
\end{bmatrix}
\end{equation*}
and
\begin{equation}\label{1x12}
\Gamma_n\Gamma_n^{-T}=
(-1)^{n+1}
\begin{bmatrix} 1&&&0
\\-2&1&&\\
&\ddots&\ddots&\\
*&&-2&1
\end{bmatrix}.
\end{equation}

If $n-m$ is odd, then
\[
(-1)^{n+1}\Delta B=2P+
P
\begin{bmatrix} 0&2&&*
\\&0&\ddots&\\
&&\ddots&2\\ 0 &&&0
\end{bmatrix}
+
\begin{bmatrix} 0&&&0
\\-2&0&&\\
&\ddots&\ddots&\\
*&&-2&0
\end{bmatrix}
P
\]
and we reduce $B$ to $0$ along its skew
diagonals starting from the upper left
corner.

If $m-n$ is even, then
\[
(-1)^{n+1}\Delta B=
-P
\begin{bmatrix} 0&2&&*
\\&0&\ddots&\\
&&\ddots&2\\ 0 &&&0
\end{bmatrix}+
\begin{bmatrix} 0&&&0
\\-2&0&&\\
&\ddots&\ddots&\\
*&&-2&0
\end{bmatrix}
P
\]
and we reduce $B$ to the form
$0^{\nwarrow}$ along its skew diagonals
starting from the upper left corner.

\subsection{Pairs of blocks ${\cal D}
(J_m(0),J_n(0))$ with $m\ge n$.}
\label{sub4}

Due to Lemma \ref{thekd}(ii), it
suffices to prove that each pair
$(B,A)$ of $n\times m$ and  $m\times n$
matrices with $m\ge n$ can be reduced
to exactly one pair of the form
\eqref{lsiu3} by adding the matrices
\begin{equation*}\label{jfr}
\Delta A=R^TJ_n(0)
+J_m(0)S,\qquad \Delta
B^T=J_m(0)^TS+
R^TJ_n(0)^T
\end{equation*}
to $A$ and $B^T$ (it is convenient for
us to reduce the transpose of $B$).

Write $S=[s_{ij}]$ and $R^T=[-r_{ij}]$
(they are $m$-by-$n$). Then
\begin{equation*}\label{drg}
\Delta A=
\begin{bmatrix}
 s_{21}-0&s_{22}-r_{11}&s_{23}-r_{12}
 &\dots&s_{2n}-r_{1,n-1}
 \\ \hdotsfor{5}\\
 s_{m-1,1}-0&s_{m-1,2}-r_{m-2,1}
 &s_{m-1,3}-r_{m-2,2}
 &\dots&s_{m-1,n}-r_{m-2,n-1}\\
 s_{m1}-0&s_{m2}-r_{m-1,1}&s_{m3}-r_{m-1,2}
 &\dots&s_{mn}-r_{m-1,n-1}
 \\ 0-0&0-r_{m1}&0-r_{m2}&\dots&0-r_{m,n-1}
 \end{bmatrix}
\end{equation*}
and
\begin{equation*}\label{drgs}
\Delta B^T=
\begin{bmatrix}
 0-r_{12}&0-r_{13}
 &\dots&0-r_{1n}&0-0
    \\
 s_{11}-r_{22}&s_{12}-r_{23}&\dots&
 s_{1,n-1}-r_{2n}
 &s_{1n}-0
   \\ \hdotsfor{5}\\
 s_{m-2,1}-r_{m-1,2}&s_{m-2,2}-r_{m-1,3}&
 \dots&s_{m-2,n-1}-r_{m-1,n}
 &s_{m-2,n}-0
 \\  s_{m-1,1}-r_{m2}&s_{m-1,2}-r_{m3}&
 \dots&s_{m-1,n-1}-r_{mn}
 &s_{m-1,n}-0
 \end{bmatrix}.
\end{equation*}
Adding $\Delta A$, we can reduce $A$ to
the form $0^{\swarrow}$; for
definiteness, we take $A$ in the form
\begin{equation}\label{gu}
0^{\downarrow}:=\begin{bmatrix}
   0_{m-1,n}
\\
  *\ *\ \cdots\ *
\end{bmatrix}.
\end{equation}
To preserve this form, we hereafter
must take
\begin{equation*}\label{VRS}
s_{21}=\dots=s_{m1}=0,\qquad
s_{ij}=r_{i-1,j-1}\quad
(2\le i\le m,\ 2\le
j\le n).
\end{equation*}
Write
\[
(r_{00},r_{01},\dots,r_{0,n-1}):=
(s_{11},s_{12},\dots,s_{1n}),
\]
then
\begin{equation*}\label{drgsu}
\Delta B^T=
\begin{bmatrix}
 0-r_{12}&0-r_{13}
 &\dots&0-r_{1n}&0-0
    \\
 r_{00}-r_{22}&r_{01}-r_{23}&\dots&
 r_{0,n-2}-r_{2n}
 &r_{0,n-1}-0
    \\
0-r_{32}&r_{11}-r_{33}&\dots&
 r_{1,n-2}-r_{3n}
 &r_{1,n-1}-0
    \\
0-r_{42}&r_{21}-r_{43}&\dots&
 r_{2,n-2}-r_{4n}
 &r_{2,n-1}-0
   \\ \hdotsfor{5}\\
0-r_{m2}&r_{m-2,1}-r_{m3}&
 \dots&r_{m-2,n-2}-r_{mn}
 &r_{m-2,n-1}-0
 \end{bmatrix}.
\end{equation*}

If $r_{ij}$ and $r_{i'j'}$ belong to
the same diagonal of $\Delta B^T$, then
$i-j=i'-j'$. Hence, the diagonals of
$\Delta B^T$ have no common $r_{ij}$,
and so we can reduce the diagonals of
$B^T$ independently.

The first $n$ diagonals of $\Delta B^T$
starting from the upper right corner
are
\begin{gather*}
0,\ \ (-r_{1n},\: r_{0,n-1}),\ \
(-\underline{r_{1,n-1}},\:
r_{0,n-2}-r_{2n},\:
\underline{r_{1,n-1}}),
   \\
(-r_{1,n-2},\: r_{0,n-3}-r_{2,n-1},\:
r_{1,n-2}-r_{3n},\: r_{2,n-1}),
    \\
(-\underline{r_{1,n-3}},\:
r_{0,n-4}-r_{2,n-2},\:
\underline{r_{1,n-3}-r_{3,n-1}},\:
r_{2,n-2}-r_{4n},\:
\underline{r_{3,n-1}}),\,\dots
\end{gather*}
(we underline linearly dependent
entries in each diagonal), adding them
we make the first $n$ diagonals of
$B^T$ as in $0^{\nevdash}$.

The $(n+1)^{\text{st}}$ diagonal of
$\Delta B^T$ is
\[
  \begin{cases}
(r_{00}-r_{22},\,r_{11}-r_{33},\,\dots,\,
r_{n-2,n-2}-r_{nn}) & \text{if $m=n$}, \\
(r_{00}-r_{22},\,r_{11}-r_{33},\,\dots,\,
r_{n-2,n-2}-r_{nn},\,r_{n-1,n-1})
& \text{if $m>n$.}
  \end{cases}
\]
Adding it, we make the
$(n+1)^{\text{st}}$ diagonal of $B^T$
equal to zero.

If $m> n+1$, then the
$(n+2)^{\text{nd}},
\dots,m^{\text{th}}$ diagonals of
$\Delta B^T$ are
\[
\begin{matrix}
(-{r_{32}},\,r_{21}-r_{43},\,
{r_{32}-r_{54}},\,
\dots,\,r_{n,n-1}),\\
\hdotsfor{1}\\
(-{r_{m-n+1,2}},\,
r_{m-n,1}-r_{m-n+2,3},\,
{r_{m-n+1,2}-r_{m-n+3,4}},\,
\dots,\,r_{m-2,n-1}).
\end{matrix}
\]
Each of these diagonals contains $n$
elements. If $n$ is even, then the
length of each diagonal is even and its
elements are linearly independent; we
make the corresponding diagonals of
$B^T$ equal to zero. If $n$ is odd,
then the length of each diagonal is odd
and the set of its odd-numbered
elements is linearly dependent; we make
all elements of the corresponding
diagonals of $B^T$ equal to zero except
for their last elements (they
correspond to the stars of ${\cal
P}_{nm}$ defined in \eqref{hui}).

It remains to reduce the last $n-1$
diagonals of $B^T$ (the last $n-2$
diagonals if $m=n$). The corresponding
diagonals of $\Delta B^T$ are
\[
\begin{matrix}
-r_{m2},\\
(-r_{m-1,2},\,
r_{m-2,1}-r_{m3}),
    \\
 (-{r_{m-2,2}},\,
r_{m-3,1}-r_{m-1,3},\,
{r_{m-2,2}-r_{m4}}),
   \\
(-{r_{m-3,2}},\,
r_{m-4,1}-r_{m-2,3},\,
{r_{m-3,2}-r_{m-1,4}},\,
r_{m-2,3}-r_{m5}),
  \\ \hdotsfor{1}
    \\
(-r_{m-n+3,2},\,
r_{m-n+2,1}-r_{m-n+4,3},\,\dots,\,
r_{m-2,n-3}-r_{m,n-1}),
\end{matrix}
\]
and, only if $m>n$,
\[
(-r_{m-n+2,2},\,
r_{m-n+1,1}-r_{m-n+3,3},\,\dots,\,
r_{m-2,n-2}-r_{mn}).
\]
Adding these diagonals, we make the
corresponding diagonals of $B^T$ equal
to zero. To preserve the zero
diagonals, we hereafter must take
$r_{m2}=r_{m4}=r_{m6}=\dots=0$ and
arbitrary
$r_{m1},\,r_{m3},\,r_{m5},\,\dots\,$.

Recall that $A$ has the form
$0^{\downarrow}$ defined in \eqref{gu}.
Since
$r_{m1},\,r_{m3},\,r_{m5},\,\dots$ are
arbitrary, we can reduce $A$ to the
form
\[
\begin{bmatrix}
   0_{m-1,n}
\\
  *\ 0\ *\ 0\ \cdots
\end{bmatrix}
\]
by adding $\Delta A$; these additions
preserve the already reduced $B$.

If $m=n$, then we can alternatively
reduce $A$ to the form
\[
\begin{bmatrix}
\hdotsfor{4}
\\ 0&0&\dots&0
 \\ *&0&\dots&0
 \\ 0&0&\dots&0
 \\ *&0&\dots&0
\end{bmatrix}
\]
preserving the form $0^{\swvdash}$ of
$B$.

\section{Off-diagonal blocks of $\cal
D$ that correspond to summands of
$A_{\text{can}}$ of distinct
types}\label{s7}

Finally, we verify the condition (ii)
of Lemma \ref{thekd} for those
off-diagonal blocks of $\cal D$
(defined in Theorem \ref{teo2}(iii))
whose horizontal and vertical strips
contain summands of $A_{\text{can}}$ of
different types.

\subsection{Pairs of blocks ${\cal D}
(H_m(\lambda),\Gamma_n)$} \label{sub9}

Due to Lemma \ref{thekd}(ii), it
suffices to prove that each pair
$(B,A)$ of $n\times 2m$ and  $2m\times
n$ matrices can be reduced to exactly
one pair of the form \eqref{lsiu4} by
adding
\[
(S^T
H_m(\lambda)+ \Gamma_n
R,\:R^T\Gamma_n
+H_m(\lambda)S),\qquad S\in
 {\mathbb C}^{2m\times n},\ R\in
 {\mathbb C}^{n\times 2m}.
\]

Taking $R=0$ and
$S=-H_m(\lambda)^{-1}A$, we reduce $A$
to $0$. To preserve $A=0$, we hereafter
must take $S$ and $R$ such that
$R^T\Gamma_n +H_m(\lambda)S=0$; that
is,
\[
S=-H_m(\lambda)^{-1}
R^T\Gamma_n.
\]
Hence, we can add
\[
 \Delta B=\Gamma_n R
-\Gamma_n^T R
H_m(\lambda)^{-T}H_m(\lambda)
\]
to $B$. Write $P=\Gamma_n^T R$, then
\[
 \Delta B=\Gamma_n
 \Gamma_n^{-T} P
-P
\left({J}_m(\lambda)
\oplus{J}_m(\lambda)^{-T}\right)
\]
Divide $B$ and $P$ into two blocks of
size $n\times m$:
\[
B=[M\ N],\qquad P=[U\
V].
\]
We can add to $M$ and $N$ the matrices
\[
\Delta M:= \Gamma_n
 \Gamma_n^{-T}U
 -U{J}_m(\lambda),\qquad
\Delta N:= \Gamma_n
 \Gamma_n^{-T}V
 -V{J}_m(\lambda)^{-T}.
\]

If $\lambda\ne(-1)^{n+1}$ then by
\eqref{1x12} the eigenvalues of
$\Gamma_n \Gamma_n^{-T}$ and the
eigenvalues of ${J}_m(\lambda)$ and
${J}_m(\lambda)^{-T}$ are distinct, and
we make $M=N=0$.

If $\lambda=(-1)^{n+1}$ then
\[
\Delta M=(-1)^{n}
\begin{bmatrix} 0&&&0
\\2&0&&\\
&\ddots&\ddots&\\
*&&2&0
\end{bmatrix}U-U \begin{bmatrix}
0&1&&0\\&0&\ddots&\\&&\ddots&1
\\ 0&&&0
\end{bmatrix},
\]
\[
\Delta N=(-1)^{n}
\begin{bmatrix} 0&&&0
\\2&0&&\\
&\ddots&\ddots&\\
*&&2&0
\end{bmatrix}V+V \begin{bmatrix}
0&&&0\\1&0&&\\&\ddots&\ddots&
\\ *&&1&0
\end{bmatrix}.
\]
We reduce $M$ to the form
$0^{\nwarrow}$ along its skew diagonals
starting from the upper left corner,
and $N$ to the form $0^{\nearrow}$
along its diagonals starting from the
upper right corner.

\subsection{Pairs of blocks ${\cal D}
(H_m(\lambda),J_n(0))$} \label{sub5}

Due to Lemma \ref{thekd}(ii), it
suffices to prove that each pair
$(B,A)$ of $n\times 2m$ and  $2m\times
n$ matrices can be reduced to exactly
one pair of the form \eqref{lsiu5} by
adding
\[
(S^T
H_m(\lambda)+ J_n(0)
R,\: R^TJ_n(0)
+H_m(\lambda)S),\qquad S\in
 {\mathbb C}^{2m\times n},\ R\in
 {\mathbb C}^{n\times 2m}.
\]

Taking $R=0$ and
$S=-H_m(\lambda)^{-1}A$, we reduce $A$
to $0$. To preserve $A=0$ we hereafter
must take $S$ and $R$ such that
$R^TJ_n(0) +H_m(\lambda)S=0$; that is,
\[
S=-H_m(\lambda)^{-1}
R^TJ_n(0).
\]
Hence we can add
\begin{align*}
 \Delta B:=& {J}_n(0) R
-{J}_n(0)^T R
H_m(\lambda)^{-T}H_m(\lambda)
        \\
=& {J}_n(0)
R-{J}_n(0)^T R
\left({J}_m(\lambda)
\oplus{J}_m(\lambda)^{-T}\right)
\end{align*}
to $B$.

Divide $B$ and $R$ into two blocks of
size $n\times m$:
\[
B=[M\ N],\qquad R=[U\
V].
\]
We can add to $M$ and $N$ the matrices
\[
\Delta M:=
{J}_n(0)U-{J}_n(0)^TU{J}_m(\lambda),\qquad
\Delta N:=
{J}_n(0)V-{J}_n(0)^TV{J}_m(\lambda)^{-T}.
\]

We reduce $M$ as follows. Let
$(u_1,u_2,\dots,u_n)^T$ be the first
column of $U$. Then we can add to the
first column $\vec b_1$ of $M$ the
vector
\begin{align*}
\Delta \vec b_1:=&
(u_2,\,\dots,\,u_n,0)^T-
\lambda
(0,u_1,\,\dots,\,u_{n-1})^T
         \\
=&
  \begin{cases}
    0 & \text{if $n=1$}, \\
(u_2,\,u_3-\lambda
u_1,\,u_4-\lambda
u_2,\,\dots,\,u_n-\lambda
u_{n-2},\,-\lambda
u_{n-1})^T & \text{if
$n>1$}.
  \end{cases}
\end{align*}
The elements of this vector are
linearly independent if $n$ is even,
and they are linearly dependent if $n$
is odd. We reduce $ \vec b_1$ to zero
if $n$ is even, and to the form
$(*,0,\dots,0)^T$ or $(0,\dots,0,*)^T$
if $n$ is odd. Then we successively
reduce the other columns transforming
$M$ to $0$ if $n$ is even and to the
form $0_{nm}^{\updownarrow}$ if $n$ is
odd.

We reduce $N$ in the same way starting
from the last column.

\subsection{Pairs of blocks ${\cal D}
(\Gamma_m,J_n(0))$} \label{sub8}

Due to Lemma \ref{thekd}(ii), it
suffices to prove that each pair
$(B,A)$ of $n\times m$ and  $m\times n$
matrices can be reduced to exactly one
pair of the form \eqref{lsiu6} by
adding
\[
(S^T\Gamma_m+
J_n(0)R,\: R^TJ_n(0)
+\Gamma_mS),\qquad S\in
 {\mathbb C}^{m\times n},\ R\in
 {\mathbb C}^{n\times m}.
\]

Taking $R=0$ and $S=-\Gamma_m^{-1}A$,
we reduce $A$ to $0$. To preserve $A=0$
we hereafter must take $S$ and $R$ such
that $R^TJ_n(0) +\Gamma_mS=0$; that is,
$ S=-\Gamma_m^{-1} R^TJ_n(0).$ Hence,
we can add
\begin{align*}
 \Delta B:&=J_n(0) R
-J_n(0)^T R \Gamma_m^{-T}\Gamma_m
       \\&
       =\begin{bmatrix}
r_{21}&\dots&r_{2m}
           \\
\hdotsfor{3}\\
r_{n1}&\dots&r_{nm}\\
0&\dots&0
\end{bmatrix}
           -(-1)^{m+1}
\begin{bmatrix}
0&\dots&0\\
r_{11}&\dots&r_{1m}
           \\
\hdotsfor{3}\\
r_{n-1,1}&\dots&r_{n-1,m}
\end{bmatrix}
\begin{bmatrix} 1&2&&*
\\&1&\ddots&\\
&&\ddots&2\\ 0 &&&1
\end{bmatrix}
\end{align*}
to $B$. We move along the columns of
$B$ starting from the first column and
reduce $B$ to $0$ if $n$ is even and to
$0^{\updownarrow}$ if $n$ is odd.

\section{Appendix: A transformation that reduces
a matrix to its miniversal deformation}
\label{sect4}

In this section, we fix  an $n\times n$
complex matrix $A$ and a $(0,\!*)$
matrix $\cal D$ of the same size such
that
\begin{equation}\label{a4nf}
{\mathbb C}^{\,n\times
n}=T(A) +
{\cal D}({\mathbb C}),
\end{equation}
in which (see \eqref{eelie} and
\eqref{a2z})
\[
T(A)=\{C^TA+AC\,|\,C\in{\mathbb
C}^{n\times n}\},\qquad
{\cal D}(\mathbb C)=
\bigoplus_{(i,j)\in{\cal
I}({\cal D})} {\mathbb
C} E_{ij},
\]
and all $E_{ij}$ are the matrix units.
We prove that the deformation $A+{\cal
D}(\vec {\varepsilon})$  of $A$ defined
in \eqref{a2z} is miniversal. To this
end, we construct an $n\times n$ matrix
${\cal S}(X)$ satisfying the conditions
(i)--(iii) of Definition \ref{dver}.

For each $P=[p_{ij}]\in\mathbb
      C^{n\times n}$, we write
\[
\|P\|:=\sqrt{\sum |p_{ij}|^2},\qquad
\|P\|_{\cal
D}:=\sqrt{\sum_{(i,j)\notin{\cal
I}(\cal D)} |p_{ij}|^2}.
\]
Note that
\begin{equation*}\label{lk}
\|aP+bQ\|\le
|a|\,\|P\|+|b|\,\|Q\|,\qquad
\|PQ\|\le \|P\|\,\|Q\|
\end{equation*}
for all $a,b\in\mathbb C$ and
$P,Q\in\mathbb C^{n\times n}$; see
\cite[Section 5.6]{hor_John}.

For every $n\times n$ matrix unit
$E_{ij}$, we fix $F_{ij}\in\mathbb
C^{n\times n}$ such that
\begin{equation}\label{8}
E_{ij}+F_{ij}^TA
+AF_{ij}\in {\cal
D}({\mathbb C})
\end{equation}
($F_{ij}$ exists by \eqref{a4nf}); we
take $F_{ij}=0_n$ if $E_{ij}\in {\cal
D}({\mathbb C})$. Write
\begin{equation}\label{kux}
a:=\|A\|,\qquad f:=\sum_{i,j}\|F_{ij}\|.
\end{equation}

For each $n\times n$ matrix  $E$, we
construct a sequence
\begin{equation}\label{rtg}
M_1:=E,\ M_2,\ M_3,\dots
\end{equation}
of $n\times n$ matrices as follows: if
$M_k=[m_{ij}^{(k)}]$ has been
constructed, then $M_{k+1}$ is defined
by
\begin{equation}\label{dei}
A+M_{k+1}:=(I_n+C_k)^T(A+M_k)(I_n+C_k)
\end{equation}
in which
\begin{equation}\label{drj1}
C_k:=\sum_{i,j}m_{ij}^{(k)}F_{ij}.
\end{equation}

In this section, we prove the following
theorem.

\begin{theorem}\label{ttft}
Given $A\in\mathbb C^{n\times n}$  and
a $(0,\!*)$ matrix $\cal D$ of the same
size that satisfy \eqref{a4nf}. Fix
$\varepsilon\in \mathbb R $ such that
\begin{equation}\label{eoj}
0<\varepsilon<\frac
1{\max \{f(a+1)(f+2),3\}}\quad
(\text{see \eqref{kux}})
\end{equation}
and define the neighborhood
\[
U:=\{E\in\mathbb C^{n\times n}\,|\,
\|E\|<\varepsilon^5\}
\]
of $0_n$. Then for each matrix $E\in
U$, the infinite product
\begin{equation}\label{gre}
{\cal
S}(E):=(I_{n}+C_1)(I_{n}+C_2)(I_{n}+C_3)\cdots
\quad (\text{see \eqref{rtg}})
\end{equation}
is convergent,
\begin{equation}\label{msu1} A+D:={\cal
S}(E)^T(A+E){\cal S}(E)\in A+{\cal
D}(\mathbb C),
\end{equation}
and
\begin{equation}\label{geo} \|{\cal
S}(E)-I_n\|< -1+(1+\varepsilon)
(1+\varepsilon ^3) (1+\varepsilon
^5)\cdots,\qquad \|D\|\le\varepsilon ^3.
\end{equation}
The matrix ${\cal S}(E)$ is a function
of the entries of $E$; replacing them
by unknowns $x_{ij}$, we obtain a
matrix ${\cal S}(X)$ that satisfies the
conditions {\rm(i)--(iii)} of
Definition \ref{dver}.
\end{theorem}

The proof of Theorem \ref{ttft} is
based on two lemmas.

\begin{lemma} \label{lem2z}
Let $\varepsilon\in\mathbb R$,
$0<\varepsilon <1/3$, and let the
sequence of real numbers
\begin{equation}\label{21z}
\delta_1,\
\tau_1,\
\delta_2,\
\tau_2,\
\delta_3,\
\tau_3,\ \dots
\end{equation}
be defined by induction:
\begin{equation*}\label{22z}
\delta_1=\tau_1=\varepsilon^5,\qquad
\delta_{i+1}
=\varepsilon^{-1}\delta_i\tau _i,\qquad
\tau_{i+1}=\tau_i
+\varepsilon^{-1}\delta_i.
\end{equation*}
Then
\begin{equation}\label{23z}
0<\delta_{i}<
\varepsilon ^{2i},\qquad
0<\tau_i<\varepsilon ^{3}\qquad\text{for all
$i=1,2,\dots$}
\end{equation}
\end{lemma}

\begin{proof}
Reasoning by induction, we assume that
the inequalities \eqref{23z} hold for
$i=1,\dots,k$. Then they hold for
$i=k+1$ since
\[
\delta_{k+1}
=\varepsilon^{-1}\delta_k \tau_k<
\varepsilon^{-1}\varepsilon ^{2k}\varepsilon ^{3}
= \varepsilon ^{2(k+1)}
\]
and
\begin{align*}
\tau_{k+1}&=\tau_k
+\varepsilon^{-1}\delta_k=
\tau_{k-1}+\varepsilon^{-1}\delta_{k-1}+
\varepsilon^{-1}\delta_k=\cdots=
\tau_{1}+\varepsilon^{-1}(\delta_{1}
+\delta_{2}+\dots+
\delta_{k})
      \\&<
\varepsilon ^{5}+\varepsilon^{-1}
(\varepsilon ^{5}
+\varepsilon^{-1} \varepsilon ^{5}\varepsilon ^{5}
+\varepsilon ^{6}+\varepsilon ^{8}
+\varepsilon ^{10}+\cdots)\\
& =\varepsilon ^{5}+\varepsilon ^{4}
+\varepsilon ^{8}+\varepsilon ^{5}
(1+\varepsilon ^{2}+\varepsilon ^{4}+\cdots)
=\varepsilon ^{5}+\varepsilon ^{4}
+\varepsilon ^{8}+\varepsilon ^{5}/(1
-\varepsilon ^{2})\\
&<\varepsilon ^{5}+\varepsilon ^{4}
+\varepsilon ^{8}+2\varepsilon ^{5}
< \varepsilon ^{4}+\varepsilon ^{8}
+\varepsilon ^{4}<3\varepsilon ^4<\varepsilon ^{3}.
\end{align*}
\end{proof}

\begin{lemma} \label{lem1}
Let $\varepsilon \in\mathbb R$ satisfy
\eqref{eoj} and let $k\in\mathbb N$.
Assume that the matrix
$M_k=[m_{ij}^{(k)}]$ from \eqref{rtg}
satisfies
\begin{equation}\label{15}
\|M_k\|_{\cal
D}<\delta_k,\quad
\|M_k\|<\tau_k\qquad (\text{see
\eqref{21z}}).
\end{equation}
Then
\begin{equation}\label{5a}
\|M_{k+1}\|_{\cal D}
<\delta_{k+1} ,\quad
\|M_{k+1}\|<\tau_{k+1}
\end{equation}
and
\begin{equation}\label{5j}
\|C_k\|<\varepsilon^{-1}\delta_k\qquad (\text{see
\eqref{drj1}}).
\end{equation}
\end{lemma}

\begin{proof}
By \eqref{8},
\[
\sum_{i,j}
m_{ij}^{(k)}E_{ij}+\sum_{i,j}
m_{ij}^{(k)}
F_{ij}^TA+\sum_{i,j}
m_{ij}^{(k)}AF_{ij}\in
{\cal D}({\mathbb C}),
\]
and so
\begin{equation}\label{18}
M_k+C_k^TA+AC_k\in
{\cal D}({\mathbb C}).
\end{equation}

If $(i,j)\in{\cal I}({\cal D})$, then
$E_{ij}\in {\cal D}({\mathbb C})$ and
$F_{ij}=0$ by the definition of
$F_{ij}$. If $(i,j)\notin{\cal I}({\cal
D})$, then $|m_{ij}^{(k)}|<\delta_k$ by
the first inequality in \eqref{15}. The
inequality \eqref{5j} holds because
\[
\|C_k\|\le
\sum_{(i,j)\notin{\cal
I}({\cal D})}
|m_{ij}^{(k)}|\|F_{ij}\|<
\sum_{(i,j)\notin{\cal
I}({\cal D})}
\delta_k\|F_{ij}\|=
\delta_k f<\delta_k\varepsilon^{-1}.
\]
By \eqref{dei} and \eqref{15},
\begin{align}\label{18de}
M_{k+1}&=M_k+C_k^T(A+M_k)+(A+M_k)C_k+C_k^T(A+M_k)C_k,
\\\nonumber \|M_{k+1}\|
&\le\|M_k\|+2\|C_k\|(\|A\|+\|M_k\|)
+\|C_k\|\|A+M_k\|\|C_k\|
 \\\nonumber &<
\tau_k+2\delta_k
f(a+\tau_k)+
\delta_kf(a+\tau_k)\delta_kf
\\\nonumber &=\tau_k+\delta_k
f(a+\tau_k)(2+
\delta_k f)
 \\\nonumber &<
\tau_k+\delta_k
f(a+1)(2+f)
< \tau_k+\delta_k\varepsilon^{-1}=\tau_{k+1}.
\end{align}
By \eqref{18} and \eqref{18de},
\begin{align*}
\|M_{k+1}\|_{\cal
D}&=\|C_k^TM_k+M_kC_k+C_k^T(A+M_k)C_k\| \\
&\le 2\|C_k\|\|M_k\|+
\|C_k\|^2(\|A\|+\|M_k\|)
   \\&<
2\delta_k
f\tau_k+(\delta_k
f)^2(a+\tau_k)
\\&< \delta_k f\tau_k (2+f(a+1))
\\&< \delta_k \tau_k f(2(a+1)+f(a+1))
<\delta_k\tau_k\varepsilon^{-1}=\delta_{k+1},
\end{align*}
which proves \eqref{5a}.
\end{proof}

\begin{proof}[Proof of Theorem \ref{ttft}]
Since $M_1=E\in U$,
$\|M_1\|<\varepsilon^5=\delta_1=\tau_1$.
Hence, the inequalities \eqref{15} hold
for $k=1$. Reasoning by induction and
using Lemma \ref{lem1}, we get
 \begin{equation*}\label{1kn}
 \|M_i\|_{\cal
D}<\delta_i,\quad
\|M_i\|<\tau_i,\quad \|C_i\|<\varepsilon^{-1}\delta_i,
\qquad i=1,2,\dots,
\end{equation*}
and by \eqref{23z}
\begin{align*}
\|C_1\|&+\|C_2\|+\|C_3\|+\cdots<
\varepsilon^{-1}(\delta_1+\delta
_2+\delta _3 +\cdots)\\&<
\varepsilon^{-1}(\varepsilon ^{2}
+\varepsilon ^{4}+\varepsilon
^{6}+\cdots) =\varepsilon(1
+\varepsilon^2+\varepsilon^4+\cdots)
\\&=\varepsilon /(1-\varepsilon ^{2})
=1/(\varepsilon^{-1}-\varepsilon)
< 1/(3-3^{-1}).
\end{align*}
The infinite product \eqref{gre}
converges to some matrix ${\cal S}(E)$
due to \cite[Theorem 4]{inf} (which
generalizes \cite[Theorem
15.14]{mark}). By \eqref{dei} and
\eqref{drj1}, the entries of each $C_i$
are polynomials in the entries of $E$.
Thus, the entries of each
\[
 {\cal
S}_k(E):=(I_{n}+C_1)(I_{n}+C_2)
\cdots(I_{n}+C_k),\qquad k=1,2,\dots,
\]
are polynomials in the entries of $E$.
Since ${\cal S}_k(E)\to {\cal S}(E)$,
the Weierstrass theorem on uniformly
convergent sequences of analytic
functions \cite[Theorem 15.8]{mark}
ensures that the entries of ${\cal
S}(E)$ are holomorphic functions in the
entries of $M$.

The inclusion \eqref{msu1} holds since
$A+M_i\to {\cal S}(E)^T(A+E){\cal
S}(E)$ and $\|M_i\|_{\cal D}
<\delta_{i}\to 0$ as $i\to\infty$.

The inequalities \eqref{geo} hold since
for each $k\in\mathbb N$ we have
\begin{align*}
\|{\cal
S}_k(E)-I_n\|&=
 \|(I_{n}+C_1)(I_{n}+C_2)
\cdots(I_{n}+C_k)-I_n\|\\
&\le\sum_{i\le k}\|C_i\|
+\sum_{i<j\le k}\|C_i\|\,\|C_j\|+\cdots\\
&\le -1+(1+\|C_1\|)(1+\|C_2\|)
(1+\|C_3\|)\cdots\\
&< -1+(1+\varepsilon^{-1}\delta_1)
(1+\varepsilon^{-1}\delta_2)
(1+\varepsilon^{-1}\delta_3)\cdots
\\&< -1+(1+\varepsilon)(1+\varepsilon^3)
(1+\varepsilon ^5)\cdots \quad\text{(by \eqref{23z})
}
\end{align*}
and
\[
M _i\to D,\quad \|M_i\|\le \tau_i<\varepsilon ^3.
\]
If $E=0_n$ then all $M_i=C_i=0_n$, and
so ${\cal S}(0_n)=I_n$.
\end{proof}

\end{document}